\documentclass[a4paper,reqno,10pt]{amsart}
\usepackage{amssymb}
\usepackage{amsmath}
\usepackage{amsthm}
\usepackage{graphicx}
\usepackage{color}
\usepackage{esint}
\usepackage{epsfig}
\usepackage{amsfonts}
\usepackage{hyperref}
\usepackage{mathtools}
\usepackage{bm}

 \allowdisplaybreaks \numberwithin{equation}{section}
\begingroup
\theoremstyle{plain}
\newtheorem{theorem}{Theorem}[section]
\newtheorem{proposition}[theorem]{Proposition}
\newtheorem{lemma}[theorem]{Lemma}
\newtheorem{corollary}[theorem]{Corollary}

\theoremstyle{definition}
\newtheorem{definition}[theorem]{Definition}
\newtheorem{remark}[theorem]{Remark}

\endgroup

\def \x {x(\cdot)}
\def \y {y(\cdot)}
\def \e {\epsilon}

\def \D {\mathbb D}
\def \Diag {\mathcal D}
\def \Diagc {\overline{\mathcal D}}

\title[On the diagram A-P-W]{About the Blaschke-Santal\'o diagram of area, perimeter and moment of inertia}
\author[A. Henrot]{Antoine Henrot}
\address[Antoine Henrot]{Universit\'e de Lorraine CNRS, IECL, F-54000 Nancy, France}
\email[Antoine Henrot]{antoine.henrot@univ-lorraine.fr}
\author[R. Gasteldello]{Rapha\"el Gastaldello}
\address[Rapha\"el Gastaldello]{Mines Nancy, Universit\'e de Lorraine CNRS, IECL, F-54000 Nancy, France}
\email[Rapha\"el Gastaldello]{raphael.gastaldello7@etu.univ-lorraine.fr}
\author[I. Lucardesi]{Ilaria Lucardesi}
\address[Ilaria Lucardesi]{Dipartimento di Matematica e Informatica "Ulisse Dini", University of Florence, Viale Morgagni 67/a, I-50134 Firenze, Italy}
\email{ilaria.lucardesi@unifi.it}

\begin{document}

\begin{abstract}
We study the Blaschke-Santal\'o diagram associated to the area, the perimeter, and the moment of inertia. We work in dimension 2, under two assumptions on the shapes: convexity and the presence of two orthogonal axis of symmetry. We discuss topological and geometrical properties of the diagram. As a by-product we address a conjecture by P\'olya, in the simplified setting of double symmetry.
\end{abstract}

\thanks{{\bf Acknowledgments.} This work was supported by the project ANR-18-CE40-0013 SHAPO financed by the French Agence Nationale de la Recherche (ANR)}

\maketitle
Date: {\today}

{\small
	
	\bigskip
	\noindent
	\bigskip
	\bigskip

\section{Introduction}
Blaschke-Santal\'o diagrams represent a powerful tool in shape optimization, to investigate the relations of triples of shape functionals $F_1,F_2,F_3$, defined in a class of admissible shapes $\mathcal A$. These diagrams consist in studying the range of the vector shape functional $(F_1,F_2):\mathcal A\to \mathbb R^2$, under the constraint $F_3(\Omega)=1$, namely the set
$$
\left\{(x,y) \in \mathbb R^2 \ :\ \exists \Omega \in \mathcal A, \ x=F_1(\Omega)\,,\ y=F_2(\Omega)\,,\ F_3(\Omega)=1\right\}.
$$
When dealing with homogeneous shape functionals, an equivalent diagram is 
$$
\left\{(x,y) \in \mathbb R^2 \ :\ \exists \Omega \in \mathcal A, \ x=F_3^\alpha(\Omega) F_1(\Omega)\,,\ y=F_3^\beta(\Omega)F_2(\Omega)\right\},
$$
where the powers $\alpha$ and $\beta$ are chosen so that $F_3^\alpha F_1$ and $F_3^\beta F_ 2$ are scale invariant.

First used by Santal\'o in \cite{Santa}, this approach has now become a standard tool in shape optimization. We cite, e.g., \cite{AH, VBP, BBF, DHP, FL, HC1, HC2, LZ}, in which shape functionals of spectral and geometric type are studied. As it appears from the literature, the theoretical analysis, even if very fine, is in general not enough for an accurate description of the diagram. Therefore it is often accompanied by numerical simulations. In this respect we cite the recent paper \cite{BBO}, in which the authors give a new and interesting way to generate random shapes whose images distribute uniformly in $\mathcal D$.  

In the present paper, we work in dimension 2 and we consider the following triple of shape functionals: the area $A(\Omega)$, the perimeter $P(\Omega)$, and the moment of inertia $W(\Omega)$ with respect to the center of gravity
$$
W(\Omega):=\int_\Omega [(x-x_G)^2 + (y-y_G)^2]\, \mathrm{d}x\, \mathrm{d}y, 
$$
where $(x_G, y_G)$ are the coordinates of the center of gravity of $\Omega$. 

The three functionals under study are invariant under rigid motion and are positively homogeneous. Therefore, as mentioned above, the Blaschke-Santal\'o diagram associated to $(A,P,W)$ 
$$
\{(x,y)\in \mathbb R^2\ :\ \exists \Omega \in \mathcal A\,,\ A(\Omega)=1\,,\ x=P(\Omega)\,,\ y=W(\Omega)\}
$$
can be deduced from the diagram
$$
\mathcal D:=\left\{ (x,y) \in \mathbb R^2\ :\ \exists \Omega\in \mathcal A\,,\ x= \frac{1}{2 \pi}\frac{A^2(\Omega)}{W(\Omega)}\,,\   y= 4 \pi \frac{A(\Omega)}{P^2(\Omega)}\right\},
$$
in which the area constraint is enclosed into the coordinate shape functionals, which are now scale invariant. The pre-factors $1/2\pi$ and $4 \pi$ are normalization constants and serve the purpose of making the diagram fit the unit square (cf. Proposition \ref{bounds}).

In the present paper we study $\mathcal D$ for the following class of planar shapes:
$$
\mathcal A:=\{\Omega \subset \mathbb R^2\ :\ \Omega\ \hbox{open, convex, with 2 orthogonal axis of symmetry}\}.
$$
The double symmetry assumption allows us to rewrite the moment of inertia as
$$
W(\Omega)=\int_\Omega (x^2 + y^2)\, \mathrm{d}x\, \mathrm{d}y,
$$
since we may assume that the two axis of symmetry coincide with the coordinate axis so that the center of gravity is located at the origin.

We prove that the diagram is simply connected and coincides with the planar region enclosed between two increasing curves which connect the point $(0,0)$ to the point $(1,1)$. The point $(0,0)$ is attained asymptotically by the image of thin shapes, whereas the point $(1,1)$ corresponds (only) to the disks. The behavior of $\mathcal D$ near these two points is investigated using the technique of shape derivatives, either of thin domains or of nearly spherical sets, and it is described by the slopes of the boundary curves in the two extremal points of the diagram.

The two boundary curves are the graphs of
$$
L^\pm(x):=\max/\min\left\{ 4\pi \frac{A(\Omega)}{P^2(\Omega)}\ :\ \Omega \in \mathcal A\,,\ \frac{1}{2\pi}\frac{A^2(\Omega)}{W(\Omega)}= x \right \}.
$$
Optimal shapes for $L^+$ turn out to be $C^{1,1}$ shapes, whereas that for $L^-$ are polygons. For values of $x$ less than $3/\pi$ we are able to characterize optimal shapes for $L^-(x)$: they are rhombi, going from the segment, in the limit as $x\to 0^+$, to the square, for $x=3/\pi$. The key argument here is the minimization of the ratio $y/x$ for $(x,y)\in \mathcal D$, namely the maximization of $F(\Omega):=P^2(\Omega)A(\Omega)/W(\Omega)$ for $\Omega \in \mathcal A$. We mention that the same issue, in the wider class of convex sets, was addressed by G. P\'olya in \cite{Polya}: according to the conjecture, still unsolved, the maximiser of $F$ should be the equilateral triangle.

For all the remaining cases, namely for $L^+(x)$, $x\in (0,1]$, and for $L^-(x)$, $x\in (3/\pi,1]$, optimizers are searched numerically. We look for the optimal shapes corresponding to points on $L^+(x)$ through their support functions. We choose to decompose the support function in Fourier series, the unknowns being the Fourier coefficients.  For the optimal shapes corresponding to points on $L^-(x)$, since we know that they are polygonal, we choose a different strategy: the unknowns being the coordinates of the vertices.

\medskip

The paper is organized as follows. In Section \ref{sec-definitions} we present the notation and we gather the first results: in Proposition \ref{simply} we show that the diagram is simply connected and its boundary is the union of two graphs of functions $L^\pm:(0,1]\to \mathbb R^2$, which have the same limit as $x\to 0^+$ and coincide (only) for $x=1$. The properties of the shapes associated to boundary points are summarized in Theorem \ref{jimmy}. The continuity and monotonicity of $L^\pm$ is investigated in the two subsequent sections, in Proposition \ref{onL-} and \ref{onL+}. Section \ref{sec-ratio} is dedicated to the study of the ratio $P^2A/W$, which allows us to give the explicit expression of $L^-(x)$ for $x\in [0,3/\pi]$. In Section \ref{sec-pentes} we analyse the diagram near its ``corners'', i.e, near $x=0$ and $x=1$. The final section is devoted to numerical shape optimization.

\section{Definitions and first properties}\label{sec-definitions}
\subsection*{Notation} Throughout the paper, we denote by $\x$ and $\y$ the coordinate shape functionals defining the diagram, that is
$$
x(\Omega):= \frac{1}{2 \pi}\frac{A^2(\Omega)}{W(\Omega)}\,,\quad  y(\Omega):= 4 \pi \frac{A(\Omega)}{P^2(\Omega)}.
$$

We use the letter $\mathbb D$ to refer to the disk of radius 1, centered at the origin.

We will also assume that the shapes of $\mathcal A$ are centered at the origin and that the coordinate system is oriented so that the axis of symmetry are the horizontal and vertical axis.

We endow the class of admissible shapes of the complementary Hausodrff distance, denoted by $d_\mathcal H$. Neighborhoods of shapes are intended with respect to the complementary Hausdorff distance.

With the symbol $\fint_\Omega f$ we denote the average of the function $f$ over the set $\Omega$, that is, $\fint_\Omega f = (\int_\Omega f )/(|\Omega|)$.

\begin{definition}\label{smalldef}
Let $\Omega$ be an admissible set. We define \textit{small deformation} of $\Omega$ a family $\{\Omega_\e\}_{\e\in [0,\e_0)}$, $\e_0>0$, with the following properties:
\begin{itemize}
\item[-] for $\e=0$ there holds $\Omega_0=\Omega$
\item[-] the map $\e \mapsto \Omega_\e$ is continuous from $[0,\e_0)$ to the class of admissible shapes $\mathcal A$, with respect to the Hausdorff distance.
\end{itemize}
\end{definition}

\subsection*{Results on $\x$ and $\y$} Here we gather some properties concerning the two functionals $\x$ and $\y$, separately. More precisely, we investigate their bounds and continuity.

\begin{proposition}\label{bounds}
The shape functionals $\x$ and $\y$ take value in the interval $(0,1]$ and the value $1$ is attained only at the disk. In particular, the diagram
$\mathcal{D}$ is contained into the square $(0,1]\times(0,1]$.
\end{proposition}
\begin{proof}
For every admissible shape $\Omega$ both $x(\Omega)$ and $y(\Omega)$ are strictly positive, implying that $\x >0$ and $\y>0$. The (sharp) upper bound on $\y$ is nothing but the classical isoperimetric inequality. As for $\x$,  this isoperimetric inequality asserting that the ball minimizes the moment of inertia is also well-known.
For sake of completeness, let us give an elementary proof: we use polar coordinates with respect to the center of gravity.
A convex shape $\Omega$ containing the origin can be described as $\Omega=\{(\rho, \theta)\ :\ \theta\in [0,2\pi],\ \rho\in [0,\rho_{max}(\theta)]\}$ for some function $\rho_{max}$. 
Computing the area and moment of inertia in polar coordinates yields:
$$
x(\Omega) 
= \frac{1}{2\pi} \frac{\left[ \int_0^{2\pi} \int_0^{\rho_{max}(\theta)} \rho \mathrm{d}\rho\,\mathrm{d}\theta \right]^2}{\int_0^{2\pi} \int_0^{\rho_{max}(\theta)} \rho^3 \mathrm{d}\rho\mathrm{d}\theta }
=\frac{1}{2\pi} \frac{ \left[ \int_0^{2\pi}  \rho_{max}^2(\theta) \mathrm{d}\theta\right]^2  }{ \int_0^{2\pi} \rho_{max}^4(\theta)\mathrm{d}\theta}.
$$
Finally, using Cauchy-Schwarz inequality, we infer that the last term is bounded above by $1$. The threshold 1 is sharp and it is attained if and only if $\rho_{max}$ is constant, namely if $\Omega$ is a disk. This concludes the proof.
\end{proof}

As already pointed out in the Introduction, the shape functionals $\x$ and $\y$ are scale invariant, since for every $t>0$
$$
P(t\Omega)= t P(\Omega),\quad  A(t \Omega) = t^2 A(\Omega), \quad W(t\Omega)= t^4 W(\Omega).
$$
This allows us to replace, if needed, the class of admissible sets with the subclass of shapes contained into the same compact set $\mathbb R^2$, obtaining the same diagram (for any choice of such a compact set), as we state in the next lemma.
\begin{lemma}\label{compactbox}
Let $K$ be a compact set. Then
$$
\mathcal D=\{(x(\Omega), y(\Omega)) \ :\ \Omega\in \mathcal A\,,\ \Omega \subset K\}\,.
$$
\end{lemma}
The advantage of working with shapes in a given box is that the classes of admissible shapes have now some compactness, in the following sense. Exploiting the compactness of the space $\{K \setminus \Omega\ :\ \Omega \in \mathcal A\}$ endowed with the Hausdorff distance (Blaschke selection theorem), we deduce that for every sequence $\Omega_n\subset K$, $\Omega_n \in \mathcal A$, there exists a subsequence (not relabeled) such that 
\begin{itemize}
\item{} either $\Omega_n\to \Omega\in \mathcal A$, with respect to the complementary Hausdorff distance,
\item{} or $\overline \Omega_n$ collapses to a segment or shrinks to a point.
\end{itemize}
\begin{proposition}\label{continuity}
The shape functionals $\x$ and $\y$ are continuous in $\mathcal A$ with respect to the complementary Hausdorff convergence.
\end{proposition}
\begin{proof}
Take $\Omega_n \to \Omega$ in $\mathcal A$, with respect to the Hausdorff convergence, as $n\to \infty$. By continuity of $A$, $P$, $W$ among convex sets, we have $A(\Omega_n) \to A(\Omega)$, $P(\Omega_n)\to P(\Omega)$, $W(\Omega_n)\to W(\Omega)$. If $\Omega$ has non empty interior, we infer that all the quantities involved (in particular $P(\Omega)$ and $W(\Omega)$) do not vanish for $n$ large enough, allowing to obtain the continuity of $\x$ and $\y$ at $\Omega$. This concludes the proof.
\end{proof}

As we have already shown in Proposition \ref{bounds}, the origin does not belong to the diagram. However, it has an important role in the description of its closure.

\begin{proposition}\label{closure}
The origin $O$ belongs to the closure of the diagram. More precisely, $\Diagc=\Diag \cup \{O\}$.
\end{proposition}
\begin{proof}
We start the proof with a direct computation of $\x$ and $\y$ on two particular families of shapes: rectangles $R_\ell$ and rhombi $S_\ell$, with semiaxis $1$ and $\ell$, with $\ell >0$. It is straightforward that:
\begin{align*}
& A(R_\ell) = 4 \ell , \quad P(R_\ell)=4(1+ \ell), \quad W(R_\ell)= \frac43 (\ell + \ell^3),
\\
& A(S_\ell) = 2 \ell , \quad P(S_\ell)=4\sqrt{1 + \ell^2}, \quad W(S_\ell)= \frac13 (\ell+\ell^3) , 
\end{align*}
so that
$$
x(R_\ell)=\frac{6}{\pi} \frac{\ell}{1+ \ell^2},  \quad y(R_\ell)=  \pi \frac{\ell}{(1 + \ell)^2},\quad x(S_\ell)=\frac{6}{\pi} \frac{\ell}{1+\ell^2}, \quad y(S_\ell)= \frac{\pi}{2}\frac{\ell}{1+\ell^2}.
$$
The two constructed families of points of the diagram converge to the origin as $\ell \to 0$ or $+\infty$. The former limit corresponds to shapes (rectangles/rhombi) collapsing to horizonral segments, the latter to shapes (rectangles/rhombi) collapsing to vertical segments. This implies that the origin is in the closure of the diagram. 

Let $\{Q_n\}\subset \Diag$ be a sequence converging to a point $Q\in \mathbb R^2$ (in norm). Let $\Omega_n \in \mathcal A$ be a sequence of associated admissible shapes. Without loss of generality (see Lemma \ref{compactbox}), we may assume that the sequence is uniformly bounded. Then two situations may occur: either a subsequence converges with respect to the complementary Hausdorff distance to some admissible shape $\Omega$, or $\overline{\Omega}_n$ collapses/shrinks to a segment/point. In the former, we infer that, by continuity of $\x$ and $\y$ in $\mathcal A$, $Q=(x(\Omega), y(\Omega))\in \Diag$. Let us examine the second situation. The case in which the chosen $\Omega_n$ shrinks to a point can be easily treated as before, since we may replace every $\Omega_n$ with a homothetic copy $\hat{\Omega}_n$, obtaining a sequence which has inradii bounded from below by a positive constant. The previous result then applies to $\hat{\Omega}_n$: we find a subsequence (not relabaled) such that $\hat{\Omega}_n\to \Omega \in \mathcal A$ and $Q_n=(x(\hat{\Omega}_n), y(\hat{\Omega}_n))\to Q=(x(\Omega), y(\Omega))\in \mathcal D$.
The last case that we have to consider is that of $\{\Omega_n\}$ thin domains, collapsing to a segment. Without loss of generality (by the scale invariance and the invariance under rotations of the functionals) may assume that $\Omega_n$ all cross the horizontal axis exactly in the segment of length 1. Moreover, they all cross the vertical axis on a centered segment of some length $\ell(n)$ which goes to 0 as $n\to \infty$. We easily obtain (by convexity) that $\Omega_n$ is contained into the centered rectangle $R_{\ell(n)}$ with sides $2$ and $2\ell(n)$ and it contains the centered rhombus $S_{\ell(n)}$ with semiaxis of length $1$ and $\ell(n)$. In particular, 
\begin{align*}
& 0< x(\Omega_n) \leq \frac{A^2(R_{\ell(n)})}{2\pi W(S_{\ell(n)})} = \frac{24}{\pi} \frac{ \ell(n)}{1+ \ell^2(n)}\to 0
\\
& 0< y(\Omega_n) \leq \frac{4\pi A(R_{\ell(n)})}{P^2(S_{\ell(n)})} = \pi \frac{\ell(n)}{(1+\ell^2(n))} \to 0.
\end{align*}
This means that the limit point is $Q=(0,0)$.
\end{proof}
\begin{proposition}\label{arcs}
The diagrams $\mathcal D$ is connected by arcs.
\end{proposition}
\begin{proof}
Let $Q_0$ and $Q_1$ be two points of $\Diag$. Take $\Omega_0$ and $\Omega_1$ two associated shapes. Consider the 1-parameter family of shapes obtained with the following Minkowski sum:
$$
\Omega_t:=t \Omega_1 \oplus (1-t) \Omega_0:=\{  t a + (1-t)b \ :\ a\in \Omega_1\,,\ b\in \Omega_0  \}, \quad t \in [0,1].
$$
The map $t\mapsto \Omega_t$ preserves convexity and it is continuous with respect to the (complementary) Hausdorff distance. We know that Hausdorff convergence is equivalent to uniform convergence of support functions (see \cite{Sch}) 
and that the support function of a Minkowski sum is the combination of support functions. If we work in $\mathcal A$, in order to preserve also the double symmetry, we have to choose the orientation of $\Omega_0$ and $\Omega_1$ so that their (two orthogonal) axis of symmetry coincide with the axis of the coordinate system. This implies that $\Omega_t$ is an admissible shape for every $t\in [0,1]$.
In view of Proposition  \ref{continuity}, we infer that the map $t\mapsto (x(\Omega_t), y(\Omega_t))$ is continuous in $\Diag$ and connects $Q_0$ to $Q_1$ with an arc. This concludes the proof.
\end{proof}
\subsection*{The upper and lower boundaries}
In order to describe the boundary of the diagram, we introduce the two functions: given $x\in (0,1]$ we set
\begin{equation}\label{defLplus}
L^+(x):=\sup \{ y(\Omega)\ :\   \Omega \in \mathcal A, x(\Omega)=x \};
\end{equation}
\begin{equation}\label{defLmoins}
L^-(x):=\inf \{ y(\Omega)\ :\   \Omega \in \mathcal A, x(\Omega)=x \}.
\end{equation}
These functions satisfy the following properties.

\begin{proposition}\label{existence}
The supremum defining $L^+$ and the infimum defining $L^-$ are attained, namely they are a maximum and a minimum, respectively. The functions $L^\pm$ coincide at $x=1$ and in the limit as $x\to 0^+$.
\end{proposition}

\begin{proof}
Let $x\in (0,1]$ be fixed. Let $y(\Omega_n)$ be a maximizing sequence for $L^+(x)$, with $\Omega_n\in \mathcal A$ and such that $x(\Omega_n) = x $ for every $n$. By Blaschke selection theorem (see also Lemma \ref{compactbox}), we may extract a subsequence (not relabeled), converging to some $\Omega \in \mathcal A$ (the fact that $x\neq 0$ ensures that the subsequence does not collapse to a segment). By continuity of $\x$ and $\y$, we infer that the limit shape satisfies $x(\Omega)=x$ and $y(\Omega)=\lim_{n\to \infty} y(\Omega_n)$. Since by assumption $\lim_{n\to \infty} y(\Omega_n)=L^+(x)$, we conclude that $\Omega$ is a maximizer. The same strategy applies for the existence of a minimizer defining $L^-(x)$.

The fact that $L^\pm$ coincide at 1 and 0 (the latter as a limit), comes from Propositions \ref{bounds}, \ref{continuity}, and \ref{closure}.
\end{proof}

We are now in a position to state the next result.

\begin{proposition}\label{simply}The diagram $\mathcal D$ is simply connected.
\end{proposition}

\begin{proof} The proof is borrowed from \cite[Theorem 3.14]{FL}, in which the authors deal with the Blaschke-Santal\'o diagram of volume, perimeter, and first Dirichlet eigenvalue of the Laplacian $\lambda_1$. Here $\lambda_1$ is replaced by the moment of inertia $W$, which is a more tractable functional. For the benefit of the reader, let us summarize the main steps of the proof.

{\it Step 1. Loops in the plane.} Let $Q_0, Q_1 \in \mathcal D$ be two points of the diagram with the same y-coordinate, say $\overline{y}\in ]0,1[$. We construct a closed (continuous) curve $\Gamma:[0,2]\to \mathbb R^2$ going first from $Q_0$ to $Q_1$ and then back from $Q_1$ to $Q_0$. In order to present the construction, let us consider $\Omega_0, \Omega_1\in \mathcal A$ associated to $Q_0$ and $Q_1$, respectively. Let $\Omega_t$, $t\in [0,1]$, denote the normalized Minkowski sum
$$
\Omega_t:= \frac{t \Omega_1 \oplus (1-t) \Omega_0}{A^{1/2}(t \Omega_1 \oplus (1-t) \Omega_0)}.
$$
Without loss of generality, up to a rotation or $\pi/2$, we may assume that the two semi-axes are chosen so that the horizontal one is greater than the vertical one. As we have already shown in Proposition \ref{arcs}, $\Omega_t\in \mathcal A$. We set
$$
\Gamma_{\Omega_0, \Omega_1}(t):=\left\{ 
\begin{array}{lll}
(x(\Omega_t), y(\Omega_t))\quad & \hbox{if }t\in [0,1],
\\
((2-t) x(\Omega_1) + (t-1)x(\Omega_0) ,\overline{y}) \quad & \hbox{if }t\in [1,2].
\end{array}
\right.
$$
The curve is clearly closed and continuous (see the proof of Proposition \ref{arcs}). For $t\in [0,1]$ the support is contained in the diagram (see Proposition \ref{arcs}), whereas for $t\in [1,2]$ it is a horizontal segment, not necessarily contained into the diagram.

We claim that the curve is contained in the stripe $\mathbb R \times [\overline{y}, \min(1, 4\overline{y})]$. This is clearly true for the horizontal part. We only need to verify that
\begin{equation}\label{claim}
\overline{y} \leq y(\Omega_t) \leq \min(1, 4\overline{y}) \quad \forall t \in [0,1].
\end{equation}

Let us prove the first inequality. Without loss of generality, we may assume that $A(\Omega_0)=A(\Omega_1)=1$. In particular $P(\Omega_0)=P(\Omega_1)=\sqrt{4\pi /\overline{y}}=:\overline{p}$. Thanks to the linearity of the perimeter for the Minkowski sum, as well as the Brunn-Minkowski inequality for the area, see \cite{Sch}, we have
\begin{align*}
& P(t \Omega_1 \oplus (1-t) \Omega_0) = t P(\Omega_0) + (1-t) P(\Omega_1)  = \overline{p},
\\
& 
A^{1/2}(t \Omega_1 \oplus (1-t) \Omega_0)\geq t A^{1/2}(\Omega_1 ) + (1-t) A^{1/2}(\Omega_0) = 1,
\end{align*}
which immediately gives the desired estimate.
Let us now prove the second inequality in \eqref{claim} (the upper bound 1 is trivial). Let $\alpha_i$ and $\beta_i$ denote the lengths of the semi-axis of $\Omega_i$, $i=0,1$. By convexity, since $\Omega_i$ contains a rhombus and is contained into a rectangle, we infer that $A(\Omega_i) \geq 2 \alpha_i \beta_i$ and $4 \sqrt{\alpha_i^2 + \beta_i^2}\leq P(\Omega_i) \leq 4(\alpha_i + \beta_i)$. In particular, recalling that by construction $\alpha_i \geq \beta_i$, $A(\Omega_i)=1$ and $P(\Omega_i)=\overline{p}$, we deduce that for $i=0,1$,
$$
\beta_i \leq \frac{1}{2\alpha_i} \leq \frac{4}{\overline{p}}, \quad \alpha_i \leq \frac{\overline{p}}{4}.
$$
The previous inequalities allow us to deduce the following upper bound on the area of the (non-normalized) Minkowski sum:
$$
A(t\Omega_1 \oplus (1-t) \Omega_0) \leq 4 (t \alpha_1 + (1-t) \alpha_0) (t \beta_1 + (1-t) \beta_0) \leq 4,
$$
so that 
$$
y(\Omega_t)=4 \pi \frac{1}{P^2(\Omega_t)} = 4 \pi \frac{A(t\Omega_1 \oplus (1-t) \Omega_0)}{{\overline{p}}^2} \leq 4 \overline{y}.
$$

{\it Step 2. Sequences of loops.} In this step we show that Hausdorff convergence of shapes entails uniform convergence of associated loops. Let us write the statement. Let $\Omega_0$ and $\Omega_1$ be two admissible shapes, different from the disk, satisfying
$$
A(\Omega_0)=A(\Omega_1)=1, \quad P(\Omega_0)=P(\Omega_1)=\overline{p}.
$$
Assume to have two sequences of admissible shapes $\{\Omega_{0,1}\}_{n\in \mathbb N}$ and $\{\Omega_{1,n}\}_{n\in \mathbb N}$ satisfying the following properties: for every $n\in \mathbb N$
$$
A(\Omega_{0,n})= A(\Omega_{1,n}) = 1, \quad P(\Omega_{0,n})=P(\Omega_{1,n}) =:p_n
$$
and, in the limit as $n\to \infty$, 
$$
\Omega_{0,n} \to \Omega_0, \quad\Omega_{1,n} \to \Omega_1
$$
with respect to the complementary Hausdorff distance. Step 1 allows us to construct a closed path associated to $\Omega_0$ and $\Omega_1$, and a family of closed paths associated to the pairs $\Omega_{0,n},\Omega_{1,n}$. For brevity, we set 
$\Gamma(t):=\Gamma_{\Omega_0, \Omega_1}(t)$ and $\Gamma_n(t):=\Gamma_{\Omega_{0,n},\Omega_{1,n}}(t)$. We claim that for every $\e>0$ there exists $n_\e\in \mathbb N$ such that 
$$
\forall n\geq n_\e\,,\ \forall t\in [0,2] \,,\ \quad \|\Gamma(t) - \Gamma_n(t) \|< \e.
$$
For $t\in [1,2]$ we have
$$
\|\Gamma(t) - \Gamma_n(t) \| \leq 4 \pi \left|  \frac{1}{\overline{p}^2} - \frac{1}{p_n^2}\right| + \frac{1}{2\pi} \left| \frac{1}{W(\Omega_1)} -\frac{1}{W(\Omega_{1,n})} \right| + \frac{1}{2\pi} \left| \frac{1}{W(\Omega_0)} -\frac{1}{W(\Omega_{0,n})} \right|.
$$
The right-hand side does not depend on $t$ and, thanks to the Hausdorff convergence of the shapes and to the continuity of the shape functionals under study, it is arbitrarily small for $n$ large enough. 
For $t\in [0,1]$ we have
\begin{equation}\label{estigamma}
\|\Gamma(t) - \Gamma_n(t) \| \leq | x(\Omega_{t}) - x(\Omega_{n,t})| + | y(\Omega_{t}) - y(\Omega_{n,t})|.
\end{equation}
Set
$$
w_n(t):= W(t\Omega_{1,n} \oplus (1-t) \Omega_{0,n}), \quad w(t):=W(t\Omega_{1} \oplus (1-t) \Omega_{0}),
$$
and 
$$
a_n(t):=A(t\Omega_{1,n} \oplus (1-t) \Omega_{0,n}), \quad a(t):=A(t\Omega_{1} \oplus (1-t) \Omega_{0}).
$$ 
Let us consider the second term in the right-hand side of \eqref{estigamma}. Using the notation above we have
\begin{align*}
| y(\Omega_{t}) - y(\Omega_{n,t})| & \leq  \frac{4 \pi a(t)}{p_n^2 \overline{p}^2}|p_n^2-\overline{p}^2| + \frac{4\pi}{p_n^2}|a(t)-a_n(t)| .
\end{align*}
The factors $a(t)$ and $p_n$ are uniformly bounded (and away from 0) in $t$ and $n$. The term $|p_n^2 -\overline{p}^2|$ is infinitesimal by assumption. The difference $a(t)-a_n(t)$ is a polynomial of degree 2 in $t$ (see the properties of Minkowski mixed volumes, e.g., in \cite{Sch}); its coefficients depend on $n$ and they go to 0 as $n\to \infty$, by the Hausdorff convergence of shapes. These estimates imply that $| y(\Omega_{t}) - y(\Omega_{n,t})| $ is arbitrarily small for $n$ large enough, uniformly in $t$.

The same strategy applies to the first term in the right-hand side of \eqref{estigamma}:
\begin{align*}
& |  x(\Omega_{t}) - x(\Omega_{n,t})|  =  \frac{1}{2\pi} \left|\frac{a^2_n(t)}{w_n(t)} - \frac{a^2(t)}{w(t)}\right| = \frac{|a^2_n(t) w(t) - a^2(t) w_n(t)|}{2\pi w_n(t) w(t)} \\
& \leq \frac{a_n(t) + a(t) }{2\pi w_n(t) }|a_n(t) - a(t)| + \frac{a^2(t) }{2\pi w_n(t) w(t)}|w(t)- w_n(t)|.
\end{align*}
As before, the factors $a_n(t), a(t), w_n(t), w(t)$ are uniformly bounded (and away from 0) in $t$ and $n$, and the difference $a_n(t) - a(t)$ is arbitrarily small, uniformly in $t$. Using polar coordinates, it is immediate to check that $|w_n(t)-w(t)|$ is bounded above by the complementary Hausdorff distance 
$$
d_n(t):=d_H(t\Omega_{1,n} \oplus (1-t) \Omega_{0,n}\ ;\  t\Omega_{1} \oplus (1-t) \Omega_{0}).
$$
Since the distance $d_H$ of two shapes coincides with the $L^\infty$ norm of the difference of the associated support functions,  and since the support function is linear for a Minkowski sum, we deduce that
$$
d_n(t) = \| t h_{1,n} + (1-t) h_{0,n} - t h_{1} - (1-t) h_0\|_\infty \leq t \|h_1 - h_{1,n}\| + (1-t) \|h_0 - h_{0,n}\|_\infty,
$$
where $h_0$, $h_{0,n}$, $h_1$, and $h_{1,n}$ denote the support functions of $\Omega_0$, $\Omega_{0,n}$, $\Omega_1$, and $\Omega_{1,n}$, respectively. Since Hausdorff convergence is equivalent to uniform convergence of support functions, we conclude that $d_n(t)$, and then $|w(t)-w_n(t)|$, are arbitrarily small for $n$ large enough, uniformly in $t$.

This concludes the proof of the uniform continuity.

{\it Step 3. Loops around holes.} We claim that $\mathcal D$ coincides with the set
$$
\{(x,y)\ :\ x\in (0,1]\,,\ y\in [L^-(x), L^+(x)] \}.
$$
This fact, together with the connectedness obtained in Proposition \ref{arcs}, implies the simple connectedness of the diagram.
Assume by contradiction that this is not true. Then there exists a point $Q=(x_Q, y_Q)\in \mathbb R^2 \setminus \Diag$ such that $x_Q\in (0,1)$ and $y_Q\in (L^-(x_Q),L^+(x_Q))$. Since $Q\neq O$, we infer that $Q\in \mathbb R^2 \setminus \left( \Diag \cup \{O\}\right) = \mathbb R^2 \setminus \overline{\Diag}$, which is an open set. Thus there exists a radius $r>0$ such that $B_r(Q)\cap \Diag = \emptyset$. Given a continuous closed curve in $\mathbb R^2$ and a point not belonging to its support, it is well defined the winding number, invariant under homotopy. Let us denote by $w_\gamma(Q)$ the winding number of a closed curve $\gamma:I \to \mathbb R^2$, being $I\subset \mathbb R$ an interval.
Given $x\in (0,1)$, we introduce the following notation
$$
\mathcal A_1 (y):=\{\Omega \in \mathcal A\ :\ A(\Omega) = 1,\ y(\Omega)=y\}.
$$
Given two shapes $\Omega_0, \Omega_1\in \mathcal A_1(y)$ we can construct, by Step 2, a continuous curve in $\mathbb R^2$, with support $\Gamma_{\Omega_0, \Omega_1}$. We notice that if $y \leq y_Q - \frac{r}{2}$, this path does not cross the point $Q$: indeed, for $t\in [0,1]$ the path is contained into $\mathcal D$, whereas for $t\in [1,2]$ it is horizontal with ordinate $y<y_Q$. Therefore, the winding number of the path around $Q$ is well defined. 

Let us consider the set 
$$
J=\left\{0 < y \leq y_Q- \frac{r}{2} \ :\ \exists \Omega_0, \Omega_1 \in \mathcal A_1(y),\ w_{\Gamma_{\Omega_0, \Omega_1}}(Q) \neq 0 \right\}.
$$
It is easy to see that $J\neq \emptyset$ (the point $y_Q-r/2 $ belongs to it) and that $J$ is bounded from below by a positive constant (by \eqref{claim}). Let $\overline{y}:=\inf J$. Two possibilities may occur: either $\overline{y}\in J$ or $\overline{y}\notin J$.
Consider, e.g., the sequence $y_n=\overline{y}-1/n$. We clearly have $y_n\notin J$. Using \cite[Corollary 3.6]{FL} (that can be easily adapted to the case of double symmetric sets), we can construct a family of pairs $(\Omega_{0,n}, \Omega_{1,n})\in \mathcal A_1(y_n)$ such that $\Omega_{0,n}\to \Omega_0$ and $\Omega_{1,n} \to \Omega_1$ with respect to the complementary Hausdorff distance. Since $y_n \notin J$, we have $w_{\Gamma_{\Omega_{0,n}, \Omega_{1,n}}}(Q)= 0$. The winding number being an invariant under homotopy, we get a contradiction. 
The second case is similar, taking a sequence $y_n \to \overline{y}^+$.
This concludes the proof.
\end{proof}

Now, using recent results of Lamboley, Novruzi, Pierre, see \cite{LNP} we prove the following

\begin{theorem}\label{jimmy}
Any convex domain, different from the disk, that minimizes $y(\Omega)$ with $x(\Omega)=x_0$ fixed is a polygon.\\
Any convex domain, different from the disk, that maximizes $y(\Omega)$ with $x(\Omega)=x_1$ fixed is $C^{1,1}$.
\end{theorem}
\begin{proof}
{\it First assertion:}
Minimizing $y(\Omega)$ with $x(\Omega)=x_0$ is equivalent, according to the scale invariance of the functionals $x$ and $y$
to 
$$\min\left\{- P(\Omega)\ \big|\  \Omega \mbox{ convex },  A(\Omega)=1, W(\Omega)=\frac{1}{2\pi x_0}=\colon w_0\right\}.$$
Now we want to use Theorem 4 in \cite{LNP} with the constraint $m(\Omega)=(0,0)$ where, using the  parametrization with
the gauge function $u$:
$$m(\Omega)=m(u)=\left(\frac{1}{2} \int_0^{2\pi} \frac{1}{u^2(\theta)} d\theta -1,  \frac{1}{4} \int_0^{2\pi} \frac{1}{u^4(\theta)} d\theta -w_0\right)$$
and the functional to minimize is
$$-P(\Omega)=-P(u)=-\int_0^{2\pi} \frac{\sqrt{u^2+{u^\prime}^2}}{u^2(\theta)} d\theta .$$
In order to apply this Theorem 4, we just need to prove ($u_0$ denotes the minimizer)
\begin{itemize}
\item that $m^\prime(u_0)$ in onto on $\mathbb{R}^2$
\item that $<m^{\prime\prime}(u_0) v,v>$ is dominated by a norm of $v$ weaker than the $H^1$-norm.
\end{itemize}
The first point comes from the fact that
$$<m^\prime(u_0),v>=\left(-\int_0^{2\pi} \frac{v}{u_0^3(\theta)} d\theta,  \;-\int_0^{2\pi} \frac{v}{u_0^5(\theta)} d\theta\right)$$
and the two linear forms are independent since $u_0$ is not constant.\\
The second point  comes from the fact that
$$<m^{\prime\prime}(u_0),v,v>=\left(3\int_0^{2\pi} \frac{v^2}{u_0^4(\theta)} d\theta,  \;5\int_0^{2\pi} \frac{v^2}{u_0^6(\theta)} d\theta\right)$$
and we can estimate both terms by the $L^2$-norm of $v$.

\medskip\noindent
{\it Second assertion:} Now we use, exactly in the same spirit, Theorem 2 in \cite{LNP}, see also Theorem 1.1 in \cite{LP}.
Now, we minimize $P(u)$ with the same constraint, and the "convexity" of the functional $u\mapsto P(u)$ provides
the good assumptions.
\end{proof}

From the previous result we immediately have the following
\begin{corollary}
The two functions $L^+$ and $L^-$ only coincide for $x=1$. 
\end{corollary}

\section{The lower boundary of the diagram}\label{sec-L-}
In this section we investigate the properties of the lower boundary of the diagram, described by the graph of the function $L^-$ (see \eqref{defLmoins}). We recall that $L^-(x)$ is defined as the minimum of the shape functional $\y$ when the shape functional $\x$ is fixed (equal to $x$). 

As we have already proved in Theorem \ref{jimmy}, the optimal shapes associated to $L^-(x)$ are polygons.

To prove the regularity and monotonicity of $L^-$, we need some preparatory results.

\begin{lemma}\label{poly1}  Let $\Omega$ be a polygon in $\mathcal A$. Then there exists $\e_0>0$ and a map
$$
[0,\e_0] \ni \e\mapsto \Omega_\e \in \mathcal A
$$
continuous with respect to the complementary Hausdorff distance, such that $\Omega_0=\Omega$ and such that $\e\mapsto x(\Omega_\e)$ is strictly increasing.
\end{lemma}
\begin{proof}
 Let $V_1$ be one of the vertexes of $\Omega$ in the first quadrant with boundary angle $\alpha \in ]0,\pi[$, and let $V_2, V_3, V_4$ be the corresponding (symmetric) points in the other three quadrants. For $\e>0$ small, let $T_\e$ denote the union of the four isosceles triangles, each of them with vertex at $V_i$, two legs lying on $\partial \Omega$, and base of length $\e$. 

We define the set $\Omega_\e$ as the set obtained by removing from $\Omega$ the four triangles, namely $\Omega_\e :=\Omega \setminus T_{\e}$. This operation clearly makes $A$ and $W$ decrease:
$$
A(\Omega_\e)= A(\Omega) \left(1- \frac{\int_{T_\e} 1}{\int_\Omega 1}\right), \quad W(\Omega_\e) = W(\Omega) \left(1- \frac{\int_{T_\e} (x^2 + y^2)}{\int_\Omega (x^2 + y^2) }\right).
$$
Therefore
\begin{align*}
x(\Omega_\e)& =x(\Omega) \left(1-2 \frac{\int_{T_\e} 1}{\int_\Omega 1} + o(\e^2) \right )\left (1+\frac{\int_{T_\e} (x^2 + y^2)}{\int_\Omega (x^2 + y^2) }+ o (\e^2) \right) 
\\ & = x(\Omega) \left(1 -2 \frac{\int_{T_\e} 1}{\int_\Omega 1}+\frac{\int_{T_\e} (x^2 + y^2)}{\int_\Omega (x^2 + y^2) }+ o (\e^2) \right).
\end{align*}
The integrals over $T_\e$ are of order $\e$ and the sign of their sum, in the limit as $\e \to 0$, gives the sign of the derivative of $\e \mapsto x(\Omega_\e)$ at $\e=0$. Therefore, if we are able to show that for $\e$ small enough there holds
$$
\fint_{\Omega} (x^2 + y^2) < \frac12 \fint_{T_\e} (x^2 + y^2),
$$
we are done. Here the symbol $\fint$ stands for the average.
Let us now make a choice on the $V_i$s: we cut four triangles near the boundary points $V_i$ satisfying the maximal distance from the origin, namely
$$
\|V_1\|^2= \|V_2\|^2 =\|V_3\|^2=\|V_4\|^2= \max_{S\in\Omega} \|S\|^2 = \max_{(x,y)\in \Omega} (x^2+y^2).
$$
Passing to the limit as $\e\to 0$ we infer that each triangle in $T_\e$ shrinks to the its vertex $V_i$ and, by the Lebesgue theorem, 
$$
\fint_{T_\e} (x^2 + y^2) \to \|V_1\|^2. 
$$
Therefore it is enough to prove that 
$$
\fint_{\Omega} (x^2 + y^2) < \frac{\|V_1\|^2}{2}.
$$
Let us use a parametrization of $\Omega$ in polar coordinates: $\theta\in [0,2\pi]$ and $\rho \in [0, \rho(\theta)]$, so that
\begin{align*}
\fint_{\Omega} (x^2 + y^2) & =\frac{ \int_0^{2\pi}\int_0^{\rho(\theta)}\rho^3(\theta)\, \mathrm{d}\rho \mathrm{d}\theta}{ \int_0^{2\pi}\int_0^{\rho(\theta)}\rho(\theta)\,\mathrm{d}\rho \mathrm{d}\theta}
\\ & = \frac12 \frac{ \int_0^{2\pi}\rho^4(\theta)\, \mathrm{d}\theta}{ \int_0^{2\pi}\rho^2(\theta)\, \mathrm{d}\theta} \leq \frac12 \max_{\theta} \|\rho(\theta)\|^2 = \frac{\|V_1\|^2}{2}.
\end{align*}
The equality holds true only if $\rho(\theta)=\|V_1\|$ for every $\theta$, namely when $\Omega$ is the disk. In particular the inequality is strict in the case under study. This concludes the proof.
\end{proof}

\begin{lemma}\label{poly2} Let $\Omega$ be a polygon in $\mathcal A$. Then there exists $\e_0>0$ and a map
$$
[0,\e_0] \ni \e\mapsto \Omega_\e \in \mathcal A
$$
continuous with respect to the complementary Hausdorff distance, such that $\Omega_0=\Omega$ and such that $\e\mapsto x(\Omega_\e)$ is strictly decreasing.
\end{lemma}
\begin{proof} The shape $\Omega_\e$ will be constructed as a small (continuous) deformation of $\Omega$. The deformation will be explicitly described in the first quadrant and repeated in the other three quadrants, in order to preserve the double symmetry.

Let $\Omega$ be a given polygon in $\mathcal A$. Let $k$ be the number of boundary vertexes falling in the first quadrant $\{x\geq 0\,,\ y\geq 0\}$. Let $\{Q_1, \ldots, Q_k\}$ be such points, ordered in counter-clockwise sense.

{\it Case 1: $k=1$ or $k=2$ and $Q_1,Q_2$ on the axis.} This two cases correspond to rectangles and rhombi. They can be described, respectively, as the 1-parameter families $\{R_\ell\}_{\ell \in ]0,1]}$ and $\{S_\ell\}_{\ell \in ]0,1]}$ defined as follows: $R_\ell$ is the rectangle associated to $k=1$ with $Q_1=(1,\ell)$; $S_\ell$ is the rhombus associated to $k=2$ with $Q_1=(1,0), Q_2=(0,\ell)$. The functions
$$
]0,1] \ni \ell\mapsto R_\ell \in \mathcal A, \quad ]0,1] \ni \ell\mapsto S_\ell \in \mathcal A
$$
are continuous with respect to the complementary Hausdorff distance, moreover (see also the proof of Proposition \ref{closure}) the maps
$$
]0,1] \ni \ell\mapsto x(R_\ell)=x(S_\ell) = \frac{6 \ell}{\pi(1+\ell^2)}
$$
are strictly increasing. This means that we can always increase/decrease $\ell$ to make $\x$ decrease in a continuous way. 

{\it Case 2: $k=2$ and either $Q_1$ or $Q_2$ on the coordinate axis.} Without loss of generality, we can assume that $Q_1$ is on the horizontal axis $y=0$, say $Q_1=(1,0)$, whereas $Q_2=(s,t)$ for some $s\in ]0,1]$ and $t>0$. In this case $\Omega$ is a hexagon with boundary points $Q_1$, $-Q_1$, $Q_2$ and the three reflections of $Q_2$ in the other three quadrants. Let us denote by $H_{s,t}$ such hexagon. As before, the map
$$
]0,1[\times ]0, +\infty[\ni (s,t)\mapsto H_{s,t}\in \mathcal A
$$
is continuous with respect to the complementary Hausdorff distance. A direct computation leads to
$$
 x(H_{s,t})=\frac{24}{\pi}\frac{t(s + 1)^2}{t^2(3s+1) + s^3+s^2+s+1}. 
$$
By computing the gradient of this function in the variables $s$ and $t$, it is immediate to check that there is no critical point in the stripe $(s,t)\in ]0,1]\times ]0, +\infty[$. As in the previous case, this implies that we can always find a direction which makes the directional derivative of $\x$ negative.

{\it Case 3: $k=2$ and neither $Q_1$ nor $Q_2$ are on the coordinate axis.} In this case $\Omega$ is an octagon. This class in $\mathcal A$ can be described as the three-parameters family $O_{s,t,u}$ associated to the two points in the first quadrant $Q_1=(1,s)$, $Q_2=(t,u)$, where $t<1$ and $u>s$. We have:
$$
x(O_{s,t,u})= \frac{24}{\pi} \frac{N(s,t,u)}{D(s,t,u)},
$$
with
\begin{align*}
N(s,t,u) = &\left[tu + (1-t)(u+s)/2\right]^2,
\\
D(s,t,u)= & t u (t^2 + 4 u^2  + t + 2) 
\\ & + (1-t)(s^3 + s^2 u + s t^2 + 2 s t + s u^2 + 3 s  + u^3 + u).
\end{align*}
Once again this function has no critical point in the region of $\mathbb R^3$ where $s,t,u$ live.

{\it Case 4: $k\geq 3$.} Here we proceed in a different way with respect to the previous cases: we will perform a parallel chord movement, consisting in making $Q_2$ slide on a line parallel to $Q_1Q_3$ passing through $Q_2$. To preserve symmetry, we do the same thing in the other three quadrants. Thanks to the standing assumptions, since $Q_2$ is not aligned to $Q_1$ vertically nor to $Q_3$ horizontally, we can perform such sliding in two directions. The resulting shape is still convex, double symmetric, and has the same area (since we keep the point $Q_2$ at the same distance from the line passing through $Q_1Q_3$).
Since this deformation does not affect the area, it is enough to study the behavior of $W$: to get the thesis we need to make $W$ increase. More precisely, since we only perturb the shape in the triangle $Q_1Q_2Q_3$, it is enough to study the behavior of $W(Q_1Q_2Q_3)$. Once clarified that we do four identical deformations in the four quadrants, we focus ourselves to what happens in one of the triangles. Since $W$ is invariant under rotation, we may assume that the segment $Q_1Q_3$ is vertical. Without loss of generality we may assume that $Q_2$ is at distance 1 from the vertical line through $Q_1Q_3$. Thus $Q_2$ is of the form $Q_2=Q_1+(1,t)$ for some positive $t$. Let $c>0$ be the distance of $Q_3$ from $Q_1$, namely $Q_3=Q_1+(0,c)$. The admissible values of $t$ depend on the sides of $\Omega$ adjacent to $Q_1Q_2$ and $Q_2Q_3$, however, given an admissible $t$, there exists a neighborhood of $t$ for which the associated shape is still admissible. Let $Q_1=(a,b)$. A direct computation shows that
\begin{align*}
W(Q_1Q_2Q_3) & = \int_0^1 \int_{tx}^{(t-c)x + c} [(x+a)^2 + (y+b)^2]\, \mathrm{d}y \, \mathrm{d}x
\\
 = &  \varphi(a,b,c) + \frac{c(c+4b)}{12} t + \frac{c}{12} t^2,
\end{align*}
for some (explicit) function $\varphi$ not depending on $t$. As a function of $t$, there is only one critical point, which is a local minimizer. This implies that for every $t$, we can always increase or decrease $t$ in such a way to make $W$ increase.
This concludes the proof.

\end{proof}

\begin{corollary}\label{deformationx}
Let $x_0\in ]0,1[$ and $\Omega_0\in \mathcal A$ with $x(\Omega_0)=x_0$ and $y(\Omega_0)=L^-(x_0)$. 
Let $\{x_n\}_{n\in \mathbb N} \subset ]0,1[$ be a sequence such that $x_n \to x_0$ as $n\to \infty$. Then there exists a sequence of shapes $\{\Omega_n\}_{n\in \mathbb N}\subset \mathcal A$ such that for $n$ sufficiently large $x(\Omega_n)=x_n$ and in the limit as $n\to \infty$ there holds $\Omega_n \to \Omega_0$ with respect to the complementary Hausdorff distance.
\end{corollary}
\begin{proof}
In view of the assumption $L^-(x_0)=y(\Omega_0)$ and Theorem \ref{jimmy} we infer that $\Omega_0$ is a polygon. We will distinguish the two cases of $x_n\to x_0^+$ and $x_n\to x_0^-$. 
Let us detail the case $x_n \to x_0^+$. In view of Lemma \ref{poly1}, there exists some $\e_0>0$ and a continuous map $[0,\e_0] \ni \e\mapsto \Omega_\e \in \mathcal A$ such that $\e\mapsto x(\Omega_\e)$ is continuous and increasing. The image of $\e\mapsto x(\Omega_\e)$ is the whole interval $[x_0, x(\Omega_{\e_0})]$. For every $n$ large enough, $x_n\in [x_0, x(\Omega_{\e_0})]$ and there exists $\e_n \in [0,\e_0]$ such that $x(\e_n)=x_n$. As $x_n \to x_0$ we have $\e_n \to 0$. The thesis then follows by taking $\Omega_n:=\Omega_{\e_n}$. The same strategy applies to the case $x_n \to x_0^-$ using Lemma \ref{poly2}.
\end{proof}

\begin{lemma}\label{pcm}
Let $\Omega \in \mathcal A$ be a polygon. Then it is not a local minimizer of $\y$.
\end{lemma}
\begin{proof}
In order to prove the statement, it is enough to show that given a polygon $\Omega \in \mathcal A$, we can construct a small deformation (see Definition \ref{smalldef}) which makes $\y$ strictly decrease. The deformation will be explicitly described in the first quadrant and repeated in the other three quadrants, in order to preserve the double symmetry.

Let $\Omega$ be a given polygon in $\mathcal A$. Let $k$ be the number of boundary vertexes falling in the first quadrant $\{x\geq 0\,,\ y\geq 0\}$. Let $\{Q_1, \ldots, Q_k\}$ be such points, ordered in counter-clockwise sense.

{\it Case 1: $k=1$ or $k=2$ and $Q_1,Q_2$ on the axis.} This two cases correspond to rectangles and rhombi. They can be described, respectively, as the 1-parameter families $\{R_\ell\}_{\ell \in ]0,1]}$ and $\{S_\ell\}_{\ell \in ]0,1]}$ defined as follows: $R_\ell$ is the rectangle associated to $k=1$ with $Q_1=(1,\ell)$; $S_\ell$ is the rhombus associated to $k=2$ with $Q_1=(1,0), Q_2=(0,\ell)$. The functions
$$
]0,1] \ni \ell\mapsto R_\ell \in \mathcal A, \quad ]0,1] \ni \ell\mapsto S_\ell \in \mathcal A
$$
are continuous with respect to the complementary Hausdorff distance, moreover (see also the proof of Proposition \ref{closure}) the maps
$$
]0,1] \ni \ell\mapsto y(R_\ell)=\frac{\pi\ell}{(1+\ell)^2}, \quad 
]0,1] \ni \ell\mapsto y(S_\ell)=\frac{\pi\ell}{2(1+\ell^2)}
$$
are strictly increasing. This means that we can always increase/decrease $\ell$ to make $\y$ decrease. 

{\it Case 2: $k=2$ and either $Q_1$ or $Q_2$ on the coordinate axis.} Without loss of generality, we can assume that $Q_1$ is on the horizontal axis $y=0$, say $Q_1=(1,0)$, whereas $Q_2=(s,t)$ for some $s\in ]0,1]$ and $t>0$. In this case $\Omega$ is a hexagon with boundary points $Q_1$, $-Q_1$, $Q_2$ and the three reflections of $Q_2$ in the other three quadrants. Let us denote by $H_{s,t}$ such hexagon. As before, the map
$$
]0,1[\times ]0, +\infty[\ni (s,t)\mapsto H_{s,t}\in \mathcal A
$$
is continuous with respect to the complementary Hausdorff distance. A direct computation leads to
$$
 y(H_{s,t})=\frac{\pi}{2}\frac{ s (1+t)}{[s + \sqrt{(1-s)^2+t^2}]^2}.
$$
By computing the gradient of this function in the variables $s$ and $t$, 
it is immediate to check that there is no critical point in the stripe $(s,t)\in ]0,1]\times ]0, +\infty[$. As in the previous case, this implies that we can always find a direction which makes the directional derivative of $\y$ negative.

{\it Case 3: $k=2$ and neither $Q_1$ nor $Q_2$ are on the coordinate axis.} In this case $\Omega$ is an octagon. This class in $\mathcal A$ can be described as the three-parameters family $O_{s,t,u}$ associated to the two points in the first quadrant $Q_1=(1,s)$, $Q_2=(t,u)$, where $t<1$ and $u>s$. We have:
$$
y(O_{s,t,u})= \frac{\pi}{2} \frac{t u +  (1-t)(u+s)/2}{[ t + s + \sqrt{(u-s)^2 + (1-t)^2}]^2}
$$
Once again this function has no critical point in the region of $\mathbb R^3$ where $s,t,u$ live.

{\it Case 4: $k\geq 3$.} Here we proceed in a different way with respect to the previous cases: we will perform a parallel chord movement, consisting in making $Q_2$ slide on a line parallel to $Q_1Q_3$ passing through $Q_2$. To preserve symmetry, we do the same thing in the other three quadrants. Thanks to the standing assumptions, since the $Q_2$ is not aligned to $Q_1$ vertically nor to $Q_3$ horizontally, we can perform such sliding in two directions. The resulting shape is still convex, double symmetric, and has the same area (since we keep the point $Q_2$ at the same distance from the line passing through $Q_1Q_3$). The deformation can be done in two senses and at least one of the two makes the perimeter of the triangle $Q_1Q_2Q_3$ (and then of the entire shape) increase, namely the functional $\y$ decreases. When the triangle $Q_1Q_2Q_3$ is isosceles, both deformations make the perimeter increase. This concludes the proof.

\end{proof}

We are now in a position to state the following.
\begin{proposition}\label{onL-}
The function $L^-:]0,1]\to \mathcal D$ is continuous and strictly increasing.
\end{proposition}
\begin{proof}
The proof is divided into three steps.

{\it Step 1: $L^-$ is l.s.c.} This property follows by construction. Let $x\in ]0,1]$ and $\{x_n\}\subset (0,1]$ an arbitrary sequence converging to $x$. We need to prove that
$$
L^-(x) \leq \liminf_n L^-(x_n).
$$
Assume (without loss of generality) that the $\liminf$ is a limit. Let $\Omega_n$ be the sequence of optimal shapes for $L^-(x_n)$. Since $x>0$, we infer that also $x_n$ are (uniformly) far from 0. This ensures the existence of a subsequence (not relabeled) $\Omega_n$ which converges to some admissible $\Omega\in \mathcal A$. By continuity $x(\Omega)=\lim_n x(\Omega_n)=\lim_n x_n =x$.  Thus by definition of $L^-(x)$ we infer that $y(\Omega) \geq L^-(x)$. Since $y(\Omega)=\lim_n y(\Omega_n)$ we conclude that
$$
L^-(x) \leq y(\Omega) = \lim_n y(\Omega_n).
$$ 

{\it Step 2: $L^-$ is u.s.c.} Take $x_0\in ]0,1]$ and take $x_n\to x_0$ satisfying
$$
\limsup_{x\to x_0} L^-(x) = \lim_{n\to \infty} L^-(x_n).
$$
Take $\Omega_0$ an optimal shape for $L^-(x_0)$, namely such that $L^-(x_0) = y(\Omega_0)$. Its existence is ensured by Proposition \ref{existence}. In view of Corollary \ref{deformationx}, there exists a sequence of  shapes $\Omega_n$ with the following properties: for every $n\in \mathbb N$
$$
\Omega_n\in \mathcal A, \quad x(\Omega_n)=x_n;
$$
moreover, in the limit as $n\to \infty$, $\Omega_n\to \Omega_0$ with respect to the complementary Hausdorff distance. In particular, $\lim_{n\to \infty} y(\Omega_n)= y(\Omega_0)$. Exploiting now the definition of $L^-$, we obtain
$$
\limsup_{x\to x_0} L^-(x) = \lim_{n\to \infty} L^-(x_n) \leq \lim_{n\to \infty} y(\Omega_n) = y(\Omega_0) = L^-(x_0).
$$

{\it Step 3: $L^-$ is strictly increasing.} By the previous steps, $L^-$ is continuous. Moreover, its infimum is 0 and it attains its maximum 1 for $x=1$. Assume by contradiction that $L^-$ is not strictly increasing. Then there exists $x_0\in ]0,1[$ local minimizer of $L^-$. Therefore there exists a neighborhood $U(x_0)$ of $x_0$ such that 
$$
\Omega \in \mathcal A, \quad x(\Omega) \in U(x_0) \quad \Rightarrow \quad L^-(x_0) \leq y(\Omega).
$$ 
Therefore any optimal shape $\Omega_0$ for $L^-(\Omega_0)$, namely such that $L^-(x_0)=y(\Omega_0)$, is a local minimizer for $\y$. In view of Theorem \ref{jimmy}, $\Omega_0$ is a polygon. The contradiction comes from Lemma \ref{pcm}, which states that a polygon can not be a local minimzer for $\y$.
\end{proof}

\section{The upper boundary of the diagram}\label{sec-L+}
In this section we investigate the properties of the upper boundary of the diagram, described by the graph of the function $L^+$ (see \eqref{defLplus}). We recall that $L^+(x)$ is defined as the maximum of the shape functional $\y$ when the shape functional $\x$ is fixed (equal to $x$). As we have proved in in Theorem \ref{jimmy}, optimal shapes for $L^+$ are $C^{1,1}$.



In order to prove the regularity and monotonicity of $L^+$, we need two preparatory results.

\begin{lemma}\label{deformationy}
Let $x_0\in ]0,1[$ and $\Omega_0\in \mathcal A$ with $x(\Omega_0)=x_0$ and $y(\Omega_0)=L^+(x_0)$. 
Let $\{x_n\}_{n\in \mathbb N} \subset ]0,1[$ be a sequence such that $x_n \to x_0$ as $n\to \infty$. Then there exists a sequence of shapes $\{\Omega_n\}_{n\in \mathbb N}\subset \mathcal A$ such that for $n$ sufficiently large $x(\Omega_n)=x_n$ and in the limit as $n\to \infty$ there holds $\Omega_n \to \Omega_0$ with respect to the complementary Hausdorff distance.
\end{lemma}
\begin{proof}
In view of the assumption $y(\Omega_0)=L^+(x_0)$ and Theorem \ref{jimmy}, we infer that $\Omega_0$ is of class $C^{1,1}$. As already done in the previous Section (cf. Lemmas \ref{poly1} and \ref{poly2}, Corollary \ref{deformationx}), it is enough to provide a deformation field $V$ localized on the strictly convex part of the boundary of $\Omega_0$ for which $y'(\Omega_0,V)\neq 0$. Then, using either $V$ or $-V$, we can make the first order shape derivative positive (or negative): the associated continuous deformation $t\mapsto \Omega_t$, for $t$ small enough, is such that $t\mapsto x(\Omega_t)$ is increasing (resp. decreasing). To conclude it is enough to take $\Omega_n:=\Omega_{t_n}$ being $t_n$ such that $x(\Omega_{t_n})= x_n$.

Let us now prove that such $V$ exists. Assume by contradiction that for every $V:\mathbb R^2 \to \mathbb R^2$ smooth, localized on the strictly convex part $\Gamma_+$ of the boundary $\partial \Omega_0$. Then, using the Hadamard's formula
\begin{align*}
x'(\Omega_0,V) & =\frac{A(\Omega_0)}{2\pi W(\Omega_0) }[2 W(\Omega_0)  A'(\Omega_0,V) - A(\Omega_0,V) W'(\Omega_0)]
\\ & = \frac{A(\Omega_0)}{2\pi W(\Omega_0) }\int_{\partial \Omega_0} \left[  2W(\Omega_0) - A(\Omega_0)\|x\|^2 \right] \, \mathrm{d}\mathcal H^1(x),
\end{align*}
we deduce that the shape under study satisfies
$$
\|x\|^2 = \frac{2W(\Omega_0)}{A(\Omega_0)} \quad \forall x \in \Gamma_+.
$$
In other words, $\Gamma_+$ is made of arcs of circle of radius $\sqrt{2W(\Omega_0)/A(\Omega_0)}$ centered at the origin. This implies that $\Omega_0$ is the disk. This is excluded by the assumptions. This concludes the proof.
\end{proof}

\begin{lemma}\label{maxloc}
Let $\Omega \in \mathcal A$ be of class $C^{1,1}$ different from the disk. Then it is not a local maximiser of $\y$.
\end{lemma}
\begin{proof}
Let $\Omega \in \mathcal A$ be of class $C^{1,1}$. The boundary $\partial \Omega$ has the following structure: 
$$\partial \Omega = \Gamma_0 \cup \Gamma_+,$$
where $\Gamma_0$ is the union of the flat parts (with $0$ curvature) and $\Gamma_+$ is the union of the strictly convex parts (with positive curvature). In the following we denote by $H(x)$ the curvature at a point $x\in \partial \Omega$. Note that $\Gamma_+$ can not be the empty set, by the regularity assumption on $\Omega$.

In the following steps we find a small deformation, acting on $\Gamma_+$ which makes $\y$ (strictly) increase. This will be done using shape derivatives, namely showing that $y'(\Omega,V)>0$ for some vector field $V:\mathbb R^2 \to \mathbb R^2$.

\smallskip

{\it Step 1. Generic shape: deformation localized on $\Gamma_+$.} Consider a smooth vector fields $V$ with support localized in a part of $\Gamma_+$. Two situations may occur: either there exists $V$ such that $y'(\Omega,V)\neq 0$ or $y'(\Omega,V)=0$ for every $V$. In the former, up to changing $V$ into $-V$, we obtain a small deformation of $\Omega$ making $\y$ increase, concluding the proof.
In the second situation, using the Hadamard's formula
\begin{align*}
y'(\Omega,V)& = \frac{4 \pi}{P^3(\Omega)} [  A'(\Omega,V) P(\Omega) - 2 A(\Omega) P'(\Omega,V)] 
\\
& = \frac{4 \pi}{P^3(\Omega)} \int_{\partial \Omega} \left[  P(\Omega) - 2 A(\Omega) H(x)    \right] V(x)\cdot n(x)\, \mathrm{d}\mathcal H^{1}(x),
\end{align*}
we deduce that the shape $\Omega$ under study satisfies the following (optimality) condition:
$$
H(x)= \frac{P(\Omega)}{2A(\Omega)}\quad \forall x \in \Gamma_+.
$$
In other words, $\Gamma_+$ is made of arcs of circle of radius
$$
R:=\frac{2A(\Omega)}{P(\Omega)}.
$$
Since $\Omega$ is not the disk, $\Gamma_+$ doe not cover the whole $\partial \Omega$, so that $\Gamma_0\neq \emptyset$. We call such domains {\it generalized stadiums}.

\smallskip

{\it Step 2. The only local maximizer of $\y$ among generalized stadiums is the disk.} Let $\Omega$ be a generalized stadium satisfying the optimality condition described in Step 1. 
Let $\{\gamma_i\}_{i=1}^k$ denote the boundary arcs, labeled in counter-clock wise sense, for some $n \in \mathbb N$. Let $C_i$ denote the center of the circle of radius $R$ to which $\gamma_i$ belongs to. Furthermore, denote by $S_{i}$, $i=1, \ldots, n$, the segment joining the arc $\gamma_i$ to the arc $\gamma_{i+1}$, with the identification $\gamma_{n+1}:=\gamma_1$.  Using the $C^{1,1}$ regularity of the boundary, we infer that every $S_i$ is tangent to the arcs $\gamma_i$ and $\Gamma_{i+1}$, in particular the segments joining $C_i$ and $C_{i+1}$ to the junction points arc-segment are orthogonal to the segment. This means that the quadrilateral obtained by considering the segment $S_i$ and the segment $C_iC_{i+1}$ is a rectangle. Repeating the same procedure to every pair arc-segment, we infer that $\Omega$ is the union of the polygon $K:=C_1C_2\ldots C_n$, $n$ circular sectors centered at $C_i$, and $n$ rectangles constructed on the sides of $K$. In other words, $\Omega$ is the Minkowski sum $\Omega=K\oplus B(0,R)$. By the classical properties of area and perimeter of Minkowski sums, we have
$$
A(\Omega)= A(K) + R P(K) + \pi R^2, \quad P(\Omega)= P(K) + 2 \pi R.
$$
Inserting these two equalities into the optimality condition relating $R$, $A(\Omega)$, and $P(\Omega)$, we get
$$
2A(K) + 2  P(K) R + 2 \pi R^2 = P(K) R + 2 \pi R^2 \quad \Rightarrow \quad 2A(K) + P(K) R = 0,
$$
implying that $K=\emptyset$. Thus the only possible shape is the disk, excluded by assumption. This concludes the proof.
\end{proof}

\begin{remark} We point out that the same result can be also deduced from \cite[Lemma 3.5]{FL}: the authors prove that the ball is the only local minimizer of $P$ among planar convex sets with area $1$. Their proof consists in comparing $P(\Omega)$ with $P(\Omega_\e)$, being $\Omega_\e$ the Minkowski sum $\Omega \oplus \e B(0,1)$ normalized with area 1.
\end{remark}

We are now in a position to state the following.

\begin{proposition}\label{onL+} The function $L^+:]0,1]\to \mathcal D$ is continuous and strictly increasing.
\end{proposition}
\begin{proof}
The proof is divided into three steps.

{\it Step 1: $L^+$ is u.s.c.} This property follows by construction. Let $x\in ]0,1]$ and $\{x_n\}\subset (0,1]$ an arbitrary sequence converging to $x$. We need to prove that
$$
L^+(x) \geq \limsup_n L^+(x_n).
$$
Assume (without loss of generality) that the $\limsup$ is a limit. Let $\Omega_n$ be the sequence of optimal shapes for $L^+(x_n)$. Since $x>0$, we infer that also $x_n$ are (uniformly) far from 0. This ensures the existence of a subsequence (not relabeled) $\Omega_n$ which converges to some admissible $\Omega\in \mathcal A$. By continuity $x(\Omega)=\lim_n x(\Omega_n)=\lim_n x_n =x$.  Thus by definition of $L^+(x)$ we infer that $y(\Omega) \leq L^+(x)$. Since $y(\Omega)=\lim_n y(\Omega_n)$ we conclude that
$$
L^+(x) \geq y(\Omega) = \lim_n y(\Omega_n).
$$ 

{\it Step 2: $L^+$ is l.s.c.} This property can be proved following the vary same strategy adopted for the u.s.c. of $L^-$ in the previous section. For the benefit of the reader, we rewrite it in this case. Take $x_0\in ]0,1]$ and take $x_n\to x_0$ satisfying
$$
\liminf_{x\to x_0} L^+(x) = \lim_{n\to \infty} L^+(x_n).
$$
Take $\Omega_0$ an optimal shape for $L^+(x_0)$, namely such that $L^+(x_0) = y(\Omega_0)$. Its existence is ensured by Proposition \ref{existence}. Moreover, by scale invariance of the shape functionals involved, we may assume that $A(\Omega_0)=1$.  In view of Lemma \ref{deformationy}, there exists a sequence of  shapes $\Omega_n$ with the following properties: for every $n\in \mathbb N$
$$
\Omega_n\in \mathcal A, \quad A(\Omega_n)=1, \quad x(\Omega_n)=x_n;
$$
moreover, in the limit as $n\to \infty$, $\Omega_n\to \Omega_0$ with respect to the complementary Hausdorff distance. In particular, $\lim_{n\to \infty} y(\Omega_n)= y(\Omega_0)$. Exploiting now the definition of $L^+$, we obtain
$$
\liminf_{x\to x_0} L^+(x) = \lim_{n\to \infty} L^+(x_n) \geq \lim_{n\to \infty} y(\Omega_n) = y(\Omega_0) = L^+(x_0).
$$

{\it Step 3: $L^+$ is strictly increasing.} By the previous steps, $L^+$ is continuous. Moreover, its infimum is 0 and it attains its maximum 1 for $x=1$. Assume by contradiction that $L^+$ is not strictly increasing. Then there exists $x_0\in ]0,1[$ local maximizer of $L^+$. Therefore there exists a neighborhood $U(x_0)$ of $x_0$ such that 
$$
\Omega \in \mathcal A, \quad x(\Omega) \in U(x_0) \quad \Rightarrow \quad L^+(x_0) \geq y(\Omega).
$$ 
This is in contradiction with Lemma \ref{maxloc}.
\end{proof}

\section{The ratio $P^2A/W$}\label{sec-ratio}
In this section, we are interested in the maximization of the ratio $P^2A/W$ among sets in the class $\mathcal{A}$.
We have two different motivations for this study:
\begin{itemize}
\item First of all, it will help drawing the Blaschke-Santal\'o diagram $\mathcal{D}$ since, maximizing 
this ratio is equivalent to minimizing $y(\Omega)/x(\Omega)$ and therefore will give a limit line below our diagram.
Moreover, we will see later that the lower part of our diagram precisely coincides with the line $y=\pi^2 x/12$
for $x\in (0, 3/\pi]$ that corresponds to all rhombi between the segment and the square (see the proof of Theorem \ref{theoF}).
\item In his paper \cite{Polya}, G. P\'olya studies this functional $F(\Omega):=P^2(\Omega) A(\Omega)/W(\Omega)$
as a characteristic example of a shape functional whose maximizer is not the disk. This is an argument against the
heuristic claim: "when the problem has many symmetries, the disk should be the optimal domain". Indeed the value
of $F$ for the disk is $8\pi^2 <96$ where $96$ is the value for all rhombi. Actually, G. P\'olya gives the conjecture
that the actual maximizer is the equilateral triangle, for which $F=108$. Our theorem below gives the optimal
domains among convex sets with two orthogonal axis of symmetry.
\end{itemize}
Let us give the main theorem of this section. We denote by $F(\Omega)$ the ratio $\displaystyle 
F(\Omega):=P^2(\Omega) A(\Omega)/W(\Omega)$.
\begin{theorem}\label{theoF}
The maximizers of $F(\Omega)$ among sets in the class $\mathcal{A}$ (i.e. convex sets with two orthogonal
axis of symmetry) are all rhombi (for example with vertices $(1,0);(0,H);(-1,0):(0;-H)$ for any $H>0$).
\end{theorem}
\begin{corollary}\label{corAH}
The Blaschke-Santal\'o diagram $\mathcal{D}$ coincides on its lower boundary $L^-(x)$
with the line $y=\pi^2 x/12$ for $x\in (0, 3/\pi]$.
\end{corollary}
Existence of a maximizer for $F$ is straightforward, using the Blaschke selection theorem and the continuity
of the geometric quantities involved for the Hausdorff convergence.
The strategy we use to prove Theorem \ref{theoF} is the following:
\begin{enumerate}
\item First we prove, using the same kind of arguments as in Theorem \ref{jimmy} that the maximizer is a polygon.
\item Then, we want to exclude vertices of the polygon that are in the interior of the triangle $T$
defined by the three points $(1,0);(1,H);(0,H)$.
For that purpose, we will not use a first order argument (using the first order optimality conditions), but a
second order argument (using the second order optimality conditions). Assuming that there is a vertex $(x_i,y_i)$
in the interior of the triangle $T$, we can move in both directions the coordinates $x_i,y_i$. We write the Hessian matrix
$\mathcal{H}_F$ of $F$ with respect to $x_i,y_i$ and we want to prove that this Hessian matrix is not negative.
\item Using an affine change of variables, we are able to consider this Hessian matrix for a triplet of points like
$(0,H);(t,uH);(1,0)$ and we can even fix the value of $H$.
\item At some point, we need to estimate global quantities like $A/P, A/P^2, A/W$ for the optimal domain.
This leads to (simpler) extremum problems that we solve analogously.
\item We conclude that there are no vertices inside the triangle by proving that the Hessian matrix is not negative. Then it remains to consider the cases of a vertex on the boundary of the triangle and eventually we are led 
just to compare the rhombus and the rectangle.
\end{enumerate}

Without loss of generality, we consider convex sets with projections $1$ and $H$ on the two axis of symmetry.
We can represent the profile of the convex set in the first quadrant by a concave function $h$ defined on $[0,1]$
with $h(0)=H$ and $h(1)=0$ and $h^\prime(0)\leq 0$.
As already mentioned, the existence of a maximizer follows by the direct methods of the calculus of variations, either working with the family of concave functions $h$ or with the convex shapes . 

{\bf 1st step}: To prove that the optimizers are polygons, we will follow the idea developed in \cite{LN}, \cite{LNP}. For that purpose,
let us write the optimality conditions in terms of $h$ and the variation $v$. First we recall the expressions of $A,P,W$ at $h$:
$$
A(h) = \int_0^1 h, \quad 
P(h)  = \int_0^1 \sqrt{1+(h')^2},
\quad 
W(h) = \int_0^1 \left(h x^2 + \frac{h^3}{3 }\right).
$$
Note that we get the total area, perimeter or moment of inertia by multiplying by 4.
A direct computation gives
\begin{align*}
\langle A'(h), v\rangle  = \int_0^1 v , \quad  
\langle P'(h), v \rangle  = \int_0^1 \frac{h' v'}{\sqrt{1+(h')^2}}, \quad 
\langle W'(h), v\rangle  = \int_0^1(h^2+x^2)v, 
\end{align*}
and
\begin{align*}
\langle A''(h) v, v\rangle  =0, \quad 
\langle P''(h) v, v\rangle  = \frac12 \int_0^1 \frac{(v')^2}{(1+(h')^2)^{3/2}}, \quad 
\langle W''(h) v ,v \rangle  = 2\int_0^1hv^2.
\end{align*}
Now, the first and second order optimality conditions (namely $F'=0$ and $F''\leq 0$) read
\begin{align}
    P' & = \frac{P}{2}\left( \frac{W'}{W}-\frac{A'}{A}\right), \notag
\\
 P'' & \leq  - \frac{P}{2A} A'' +
 \frac{ P}{2 W} W''
+  \frac{P}{A^2} (A')^2 .\label{stima}
\end{align}
Using in \eqref{stima} the expressions of $A$,$P$,$W$, $A'$, $P'$, $W'$, $A''$, $P''$, $W''$, we obtain
$$
 \frac12 \int_0^1 \frac{(v')^2}{(1+(h')^2)^{3/2}} \leq 
 \frac{ P}{W} \int_0^1hv^2
+  \frac{P}{A^2} \left( \int_0^1 v\right)^2. 
$$
Let $I_v$ denote an interval containing the support of $v$: using the Cauchy-Schwartz inequality, we get
$$
\frac12 \int_0^1 \frac{(v')^2}{(1+(h')^2)^{3/2}} \leq \left[
 \frac{ P}{W}\, \|h\|_{L^\infty(I_v)} +  \frac{P}{A^2}\,
 |I_v|^2 \right]  \|v \|_{L^2(I_v)}^2.
$$
This estimate is true for every admissible $v$ with the properties above. This will be crucial for the proof that $h$ is polygonal.

In order to prove the polygonal structure, we follow \cite{LN}.  We want to prove that the support of the measure 
$h''$ is discrete. Assuming, this is not the case: it contains an accumulation point $x_0$ and, 
for any $\varepsilon_n>0$ (or $\varepsilon_n<0$)
we  can find at least four point $x_1^n< x_2^n<x_3^n <x_4^n$ in the interval $[x_0,x_0+\varepsilon_n]$
such that the support of $h''$ satisfies
$$supp^t(h'')\ \cap (x_j^n,x_{j+1}^n) \not= \emptyset,\quad \mbox{ for }  j=1,2,3.$$
Then we construct three functions $v_{n,i}$ with support in $[x_0,x_0+\varepsilon_n]$ in the following way: 
$v_{n,i}$ solves the EDO : $v''_{n,i}=\chi_{(x_i^n,xi_{i+1}^n)}.h''$ and $v_{n,i}=0$ in $(0,\varepsilon_n)^c$, $i=1,3$. \\
Now, we choose three constants $\alpha_1, \alpha_2, \alpha_3$ such that the function $v_n$ defined as
$$v_n = \sum_{i=1}^3 \alpha_i v_{n,i}$$
satisfies $v^\prime_n(x_0)=v^\prime_n(x_0+\varepsilon_n)=0$ and then $v_n$ (extended by zero) has its support
in $[x_0,x_0+\varepsilon_n]$ and satisfying $v_n''= \sum_{i=1}^3 \alpha_i \chi_{(x_i^n,xi_{i+1}^n)}.h''$
is admissible as a perturbation of the optimum $h$.

Then we use the previous inequality, with $I_v=[x_0,x_0+\varepsilon_n]$:
$$
\frac12 \int_0^1 \frac{(v_n')^2}{(1+(h')^2)^{3/2}} \leq \left[
 \frac{ P}{W}\, \|h\|_{L^\infty(I_v)} +  \frac{P}{A^2}\,
 |I_v|^2 \right]  \|v_n \|_{L^2(I_v)}^2,
$$
and we can assume that there exist $0 \leq C < +\infty $ such as $|h'|\leq C$ on $I_v$.
\\
So we have 
$$
\frac{1}{2(1+C^2)^{\frac32}} \|v_n' \|_{L^2(I_v)}^2\leq \left[
 \frac{ P}{W}\, \|h\|_{L^\infty(I_v)} +  \frac{P}{A^2}\,
 |I_v|^2 \right]  \|v_n \|_{L^2(I_v)}^2.
$$
By using the Poincaré's inequality on $I_v$, we have:
$$
\frac{\pi^2}{\varepsilon_n^2} \leq 2(1+C^2)^{\frac32} \left[\frac{P}{W}\, \|h\|_{L^\infty(I_v)} +  \frac{P}{A^2}\,
 |I_v|^2 \right]
$$
This gives us the contradiction since the left part tends to $+\infty$ when $n$ increases and the right part is
bounded. 

\begin{remark}
Note that the sole admissible $v$ for $h$ linear is $v\equiv 0$. Therefore the optimality conditions are (trivially) satisfied.
\end{remark}

In view of the previous proposition, we infer that a maximizer $h$ satisfies either $h''\equiv 0$ or $h''= \sum_{i=1}^k \alpha_i \delta_{x_i}$ in the open set $(0,1)$ for a finite family of $\alpha_i\in \mathbb R$ and $x_i \in (0,1)$.
Our purpose is now to exclude this second case.

\medskip\noindent
{\bf 2nd step}
Let us assume that the optimal polygon has a "free" vertex $M_1=(x_1,y_1)$. By free, we mean that we can move infinitesimally $x_1$ and $y_1$ in any direction keeping an admissible (convex) polygon. Taking the first interior vertex, we can assume that his two neighbouring vertices
are $M_0=(0,H)$ and $M_2=(x_2,y_2)$. 
To compute the successive derivatives $\partial F/\partial x_1; \partial F/\partial y_1; \partial^2 F/\partial x_1^2;
\partial^2 F/\partial x_1 \partial x_2; \partial^2 F/\partial x_2^2$, we just need to look at the contribution of $x_1,y_1$
in the global expression of $P,A,W$. Let us denote by $\widehat{P}, \widehat{A}, \widehat{W}$ the remaining parts not
depending on $x_1,y_1$. We have
$$
\begin{array}{l}
P=\widehat{P} + 4\left(x_1^2+(y_1-H)^2\right)^{1/2} + 4\left((x_2-x_1)^2+(y_2-y_1)^2\right)^{1/2} \\
A=\widehat{A} + 2(x_1(H-y_2) + x_2(y_1+y_2)) \\
W=\widehat{W}+4(I_1+J_1+I_2+J_2)
\end{array}
$$
where, denoting by $h_1(x), h_2(x)$ the expression of $h(x)$ on the first and the second interval, namely
$$
\begin{array}{c}
h_1(x)=\frac{y_1-H}{x_1}\,x+H,\quad x\in [0,x_1] \\
h_2(x)=\frac{y_2-y_1}{x_2-x_1}\,(x-x_2)+y_2,\quad x\in [x_1, x_2] \\
\end{array}
$$
we have
$$
\begin{array}{l}
I_1=\int_0^{x_1} x^2 h_1(x) = \frac{1}{4} y_1 x_1^3 +\frac{1}{12} x_1^3 H \\
I_2=\int_{x_1}^{x_2} x^2 h_2(x) = \frac{1}{4} (y_2 x_2^3 -y_1 x_1^3) -\frac{1}{12} (x_1^2 +x_1x_2 +X_2^2) (y_2x_1-y_1x_2)\\
J_1= \int_0^{x_1} \frac{1}{3} h_1^3 (x)= \frac{1}{12} x_1 (y_1^3+y_1^2 H+y_1 H^2 +H^3) \\
J_2= \int_{x_1}^{x_2} \frac{1}{3} h_2^3 (x)= \frac{1}{12} (x_2-x_1) (y_1^3+y_1^2 y_2+y_1 y_2^2 +y_2^3)
\end{array}
$$
Using the previous formulas, we can compute the first and second derivatives of $P,A,W$ 
with respect to $x_1,y_1$ and therefore, the derivatives of $F=P^2A/W$.
Now, the first order optimality condition for $F$ writes
$$\frac{1}{W}\frac{\partial W}{\partial x_1}=\frac{2}{P}\frac{\partial P}{\partial x_1} + \frac{1}{A}\frac{\partial A}{\partial x_1}$$
and the same for the derivative in $y_1$. We use these relations to simplify the computations of the second derivative.
For example, $\partial^2 F/\partial x_1^2$ can be written
$$\frac{\partial^2 F}{\partial x_1^2}=F\left(\frac{2}{P}\frac{\partial^2 P}{\partial x_1^2} + 
\frac{2}{P^2}\left(\frac{\partial P}{\partial x_1}\right)^2 + \frac{4}{AP}\frac{\partial P}{\partial x_1} \frac{\partial A}{\partial x_1} - \frac{1}{W}\frac{\partial^2 W}{\partial x_1^2}
\right)$$
(we use here the fact that the second derivative $\partial^2 A/\partial x_1^2$ vanishes). Similarly for the other derivatives.

Let us give the final expression of the Hessian matrix after a straightforward computation.
For sake of simplicity, we introduce the two angles $\theta_1,\theta_2$ that the segments $M_0M_1$ and $M_1M_2$ make
with the horizontal, namely
$$\theta_1=\arctan\left(\frac{H-y_1}{x_1}\right), \quad \theta_2=\arctan\left(\frac{y_1-y_2}{x_2-x_1}\right),$$
and the relative coordinates
$$t=\frac{x_1}{x_2}, \quad u=\frac{y_1-y_2}{H-y_2}.$$
We get
\begin{align*}
A \frac{\partial^2 F}{\partial x_1^2} = & \frac{8A}{P}\left(\frac{\sin^2\theta_1}{M_0M_1} + \frac{\sin^2\theta_2}{M_1M_2}\right) +\frac{32 A}{P^2} (\cos\theta_1 - \cos\theta_2)^2 \\
& +  \frac{32 (H-y_2)}{P} (\cos\theta_1 - \cos\theta_2)
-\frac{2 A (H-y_2)}{3W } (3x_1+x_2 u)
\\
A \frac{\partial^2 F}{\partial x_1\partial y_1} =& \frac{8A x_2(H-y_2)}{P}
\left(\frac{t(1-u)}{M_0M_1^3} + \frac{u(1-t)}{M_1M_2^3}\right) \\
& + \frac{32 A}{P^2} (\cos\theta_1 - \cos\theta_2) (\sin\theta_2 - \sin\theta_1) \\
& + \frac{16}{P} \left(x_2(\cos\theta_1 - \cos\theta_2)+(H-y_2) (\sin\theta_2 - \sin\theta_1)\right)\\
& - \frac{A}{3W } [x_2(x_2+2x_1)+(H-y_2)(H+y_2+2y_1)]
\\
A \frac{\partial^2 F}{\partial y_1^2} =& \frac{8A}{P}\left(\frac{\cos^2\theta_1}{M_0M_1} + \frac{\cos^2\theta_2}{M_1M_2}\right) +\frac{32 A}{P^2} (\sin\theta_2 - \sin\theta_1)^2 \\
& + \frac{32 x_2}{P} (\sin\theta_2 - \sin\theta_1)
-\frac{2 A x_2}{3W } \left[4H-(H-y_2)((1-t)+3(1-u))\right].
\end{align*}
Now, let us consider the following affine transformation given by  the change of variable
$$
x'=\frac{x}{x_2},\qquad y'=\frac{y-y_2}{4(H-y_2)}.
$$
It has the effect of transforming the point $M_0$ in $N_0=(0,\frac{1}{4})$, the point $M_2$ in $N_2=(1,0)$ and the point $M_1$ in $N_1=(x'_1,y'_1)$
in such a way that these three points remain in a "convex" position
(we can assume $H>y_2$,  otherwise $M_1$ would not be a vertex).
Moreover, when computing the new Hessian matrix inherited with this change of variable, we see that it has the same properties of the original Hessian matrix,
for example the determinants of the two Hessian matrices are equal up to the positive factor $1/16 x_2^2 (H-y_2)^2$. Therefore, we can restrict
to this particular situation with the three points $((0,\frac{1}{4}); (t,u/4); (1,0)$ for which the above formula simplifies. In particular, we will be interested in the following
quantity $(1,-1) \mathcal{H}_F (1,-1)^T$ or
\begin{equation}\label{quantD}
\mathcal{E}:= \frac{\partial^2 F}{\partial x_1^2} +  \frac{\partial^2 F}{\partial y_1^2} -2  \frac{\partial^2 F}{\partial x_1 \partial y_1}
\end{equation}
together with the trace of the Hessian matrix
$$
\mathcal{T}:= \frac{\partial^2 F}{\partial x_1^2} +  \frac{\partial^2 F}{\partial y_1^2} .
$$

\medskip\noindent
{\bf 3rd step}
In the second derivatives of the functional $F$, appear some global quantities involving $A,P,W$. In order to be able to prove that  the quantity
$\mathcal{E}$ is positive, we need to estimate $A,P,W$ and some ratios. This is the aim of the following proposition.
Without loss of generality, we will use the following normalization: we work in the subclass:
$$\mathcal{A}_0=\{\Omega \in \mathcal{A},  h(0)=H, h(1)=0\}$$
where we assume $H\leq 1$.
\begin{proposition}\label{propesti}
Let $\Omega$ in $\mathcal{A}_0$, then we have the following inequalities
$$A(\Omega)\geq 2H, \quad P(\Omega) \leq 4(1+H),\quad \frac{A(\Omega)}{P(\Omega)}\,\geq \frac{H}{2\sqrt{1+H^2}}$$
$$\frac{A(\Omega)}{P^2(\Omega)}\,\geq \frac{H}{8(1+H^2)},\quad \frac{A(\Omega)}{W(\Omega)}\,\leq \frac{6}{1+H^2}\,.$$
\end{proposition}
\begin{proof}
The two first inequalities come immediately from the fact that $\Omega$ contains the rhombus
of vertices  $(\pm 1,0);(0,\pm H)$ and is contained in the rectangle $(-1,1)\times (-H,H)$.
To prove the three other inequalities, we solve  a shape optimization problem in the class $\mathcal{A}_0$.
Existence of an optimal domain is immediate each time.

\medskip 
{\it Minimizing $A/P$.} Let $\Omega^*$ be a minimizer. Working exactly as the beginning of the proof of Theorem \ref{theoF}, we infer that 
$\Omega^*$ is a polygon. Let us assume that $\Omega^*$ contains a free vertex (namely a vertex $(x_1,y_1)$ with $0<x_1<1, 1-x_1<y_1<H$).
Let us write the optimality conditions we obtain by moving this vertex.
On the one-hand, the first optimality condition in $x_1$ can be written
$$\frac{1}{A}\,\frac{\partial A}{\partial x_1 }\,- \frac{1}{P}\,\frac{\partial P}{\partial x_1 } =0.$$
Let us now compute the second derivative in $x_1$. Taking into account that $\frac{\partial^2 A}{\partial x_1 ^2}\ =0$, we get
$$\frac{\partial^2}{\partial x_1^2} \left(\frac{A}{P}\right) = - \frac{1}{A^2}\left(\frac{\partial A}{\partial x_1 }\right)^2 + 
 \frac{1}{P^2}\left(\frac{\partial P}{\partial x_1 }\right)^2 -  \frac{1}{P}\,\frac{\partial^2 P}{\partial x_1^2}\,.$$
 Now, using the first order optimality condition, the two first terms cancel out and we get
$$\frac{\partial^2}{\partial x_1^2} \left(\frac{A}{P}\right) = -  \frac{1}{P}\,\frac{\partial^2 P}{\partial x_1^2}\,=-\frac{(y_1-y_0)^2}{A_0A_1^3}
-\frac{(y_1-y_2)^2}{A_2A_1^3} $$
the second derivative in $x_1$ being strictly negative, we get a contradiction with the minimality. Therefore, $\Omega^*$ cannot have a free vertex and
the only possibilities that remain to be considered are
\begin{itemize}
\item the rhombus
\item the rectangle
\item another vertex on the horizontal line $y=H$
\item or/and another vertex on the vertical line $x=1$.
\end{itemize}
Actually in the two last cases, we can move freely the vertex along the horizontal (or the vertical) line, therefore the previous computation with the
variable $x_1$ (or the variable $y_1$) still holds and leads to a similar contradiction. Thus it remains only to compare the rectangle and the rhombus
for which the ratio $A/P$ equals respectively $H/(1+H)$ and $H/2\sqrt{1+H^2}$ and a direct comparison shows that the rhombus gives the lower value.

\medskip
{\it Minimizing $A/P^2$.} We proceed exactly in the same way.
We get for the second derivative in $x_1$:
$$\frac{\partial^2}{\partial x_1^2} \left(\frac{A}{P^2}\right) = - \frac{1}{2A^2}\left(\frac{\partial A}{\partial x_1 }\right)^2 
- \frac{2}{P}\,\frac{\partial^2 P}{\partial x_1^2}$$
and the right-hand side being negative, we conclude in the same way that we just have to compare the rhombus and the rectangle.
For the rhombus, the ratio $A/P^2$ is $H/8(1+H^2)$ while for the rectangle, it is equal to $H/4(1+H)^2$ 
and the result follows from the comparison of these two numbers.

\medskip
{\it Maximizing $A/W$ or minimizing $W/A$.} This case is more complicated.  The fact that the perimeter is
not in the functional makes not clear whether the maximizer is a polygon.
Let us write the optimality condition, using the formalism developed in \cite{LN} to take into account
the concavity constraint on the function $h$ describing the boundary of the optimal set.
Since the derivative of the area and the moment of inertia are
$$
\langle A'(h),v\rangle=\int_0^1 v \mathrm{d}x \qquad \langle W'(h),v\rangle=\int_0^1 (x^2 +h^2)v \mathrm{d}x$$
the first order optimality condition writes: there exists a function $\xi$ in $H^1(0,1)$, $\xi \geq 0$,
$\xi=0$ on $S$ the support of the measure $h^{\prime\prime}$ such that
\begin{equation}\label{optxi}
- \frac{W}{A}\, \xi^{\prime\prime}= h^2+x^2 - \frac{W}{A}\,.
\end{equation}
The right-hand side of \eqref{optxi} being continuous, we see that the function $\xi$ is indeed in $C^2(0,1)$.
Now, the support $S$ is closed: let us assume that there exists a (maximal) open interval $(\alpha,\beta)$ in its
complement with $0<\alpha<\beta<1$.  On this interval, we have
$$- \frac{W}{A}\, \xi^{\prime\prime}= h^2+x^2 - \frac{W}{A} ,\;x\in (\alpha,\beta)$$
with the boundary conditions $\xi(\alpha)=\xi(\beta)=0$. Moreover, since $\xi$ is $C^2$ and $\xi \geq 0$,
we must have also $\xi^\prime(\alpha)=\xi^\prime(\beta)=0$.  Now, since $h^{\prime\prime}=0$ on $(\alpha,\beta)$
we see that $h$ is affine on this interval and therefore the right-hand side of the equation \eqref{optxi} is
a polynomial of degree 2. Taking into account the boundary conditions, this implies that
$\xi(x)=T(x-\alpha)^2(x-\beta)^2$ on $(\alpha,\beta)$.  Now, coming back to the equation, we see that the term
in $x^2$ of $\xi^{\prime\prime}$ should be, on the one hand, equal to $12T$ and, on the other hand, negative (because it is positive for $h^2+x^2$). The negativity of $T$ is in contradiction with $\xi \geq 0$.
In conclusion, the complement of $S$ cannot have an internal open interval. In other words, we are led to
the following possibilities
$$S^c=(0,\alpha) \cup (\beta,1)\quad \mbox{ with $0\leq \alpha \leq \beta \leq 1 $}.$$
Let us start with the case $ 0< \alpha < \beta < 1 $ that means that $h$ is affine on the intervals
$[0,\alpha] \cup [\beta,1]$ and is strictly convex for $x\in (\alpha,\beta)$.  On this strictly convex part,
since here $\xi=0$, we have $h^2+x^2=W/A$: so the boundary is an arc of circle of radius $R=\sqrt{W/A}$.
We parametrize the boundary with two angles $\theta_1,\theta_2$ such that (see Figure)
\begin{center}
\begin{figure}[h]
\includegraphics[scale=0.2]{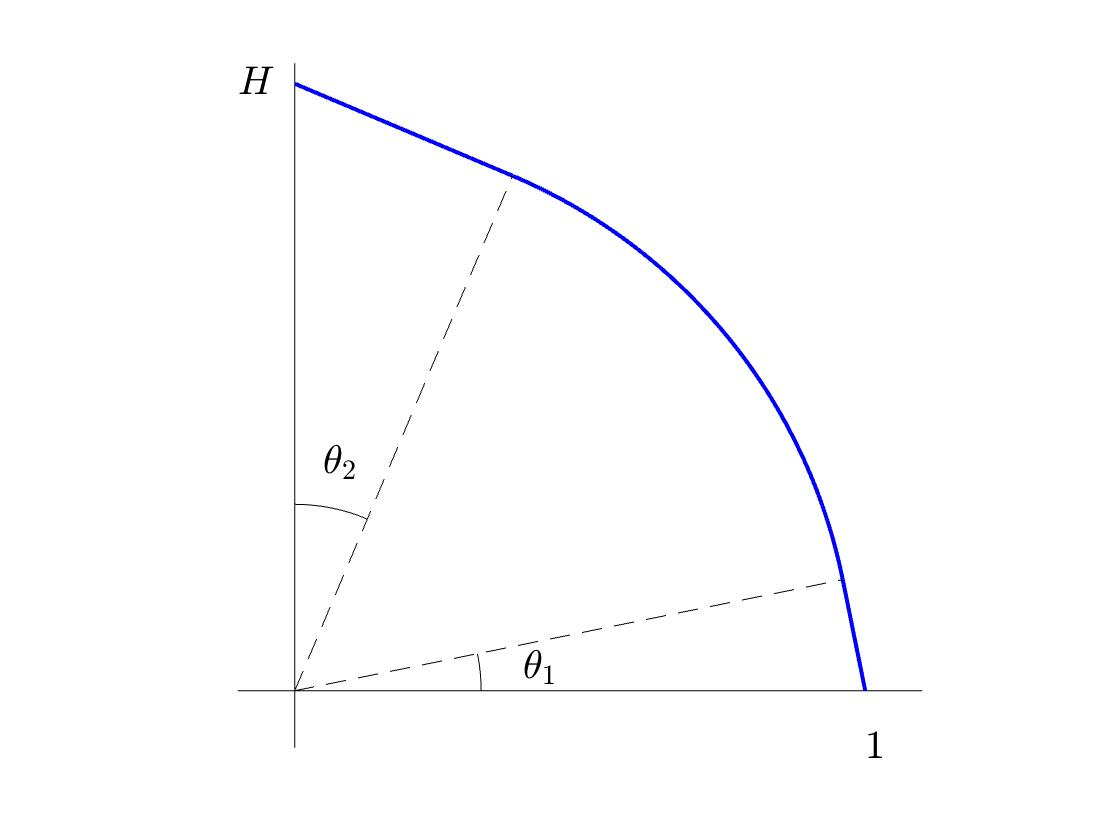}
\end{figure}
\end{center}

\begin{itemize}
\item the first segment on $[0,\alpha]$ is given in polar coordinates as $\rho=R/\sin(\theta+\theta_2)$, 
$\theta\in [\pi/2-\theta_2,\pi/2]$,
\item the arc of circle is given by $\rho=R, \theta \in [\theta_1, \pi/2-\theta_2]$,
\item the second segment on $[\beta,1]$ is given in polar coordinates as $\rho=R/\cos(\theta-\theta_1)$, 
$\theta\in [0,\theta_1]$.
\end{itemize}
We also have the relations $R=\sqrt{\frac{W}{A}}=\cos\theta_1$ and $H=R/\cos\theta_2$.
Using the expressions in polar coordinates, we immediately get
$$A=2R^2\left(\tan\theta_1 + \frac{\pi}{2} -\theta_2-\theta_1 + \tan\theta_2\right),$$
$$W=R^4\left(\int_0^{\theta_1}\frac{dt}{\cos^4(t)} +\frac{\pi}{2} -\theta_2-\theta_1  +
\int_0^{\theta_2}\frac{dt}{\cos^4(t) }\right).$$
If $0<\theta_1$ and $\theta_1+\theta_2 < \pi/2$, we can differentiate with respect to $\theta_1$ and 
we should get 0 for this derivative.
Now
$$\frac{\partial}{\partial \theta_1}\left(\frac{W}{A}\right)=\frac{R^4}{A}\left(\frac{1}{\cos^4\theta_1}-1\right)
-\frac{2WR^2}{A^2}\left(\frac{1}{\cos^2\theta_1}-1\right)$$
Using $\cos\theta_1=R$ and $R^2=W/A$, this derivative can be written
$$\frac{\partial}{\partial \theta_1}\left(\frac{W}{A}\right)=\frac{(1-R^2)^2}{A}.$$
But if $R$ would be equal to 1, we would have $\theta_1=0$ that is not the assumption. This computation shows that
the case $0<\theta_1$ and $\theta_1+\theta_2 < \pi/2$ cannot occur. We can prove exactly in the same way that 
$0<\theta_2$ and $\theta_1+\theta_2 < \pi/2$ cannot occur.  Therefore, it remains only two cases:
\begin{itemize}
\item either $\theta_1=0$ and $\theta_2=0$ that corresponds to the unit disk for which $W/A=R^2/2$,
\item or $\theta_1+\theta_2 = \pi/2$ that corresponds to a polygon with one interior vertex, say $(x_1,y_1)$.
\end{itemize}
In that case, it is more convenient to express all the quantities in the cartesian coordinates.
Let us introduce $t=x_1$ and $u=y_1/H$ both in the interval $[0,1]$ The concavity assumption is
$t+u\geq 1$. With these variable, we have
$$A=2H(t+u)$$
$$W=\frac{H}{3}\left[t^3 +(1+t+t^2)u+H^2(t(1+u+u^2)+u^3)\right].
$$
As previously, we want to prove that the vertex $(x_1,y_1)$ cannot be "free".
This leads us  to study the function with two variables $t,u$
$$\psi(t,u)=\frac{t^3 +(1+t+t^2)u+H^2(t(1+u+u^2)+u^3)}{t+u},$$
and prove that it cannot have a minimizer in the open set $\{0<t<1, 1-t<u<1\}$. Let us compute the derivative
of $\psi$ with respect to $u$:
$$(t+u)^2 \frac{\partial \psi}{\partial u}=t+t^2+H^2\left(t(4u^2-1)+t^2(1+2u)+2u^3\right).$$
Since $t+H^2t(4u^2-1)>t(1-H^2)\geq 0$, we see that the right-hand side is positive and cannot vanish.
This means that $u$ has to be equal either to $1-t$, that corresponds to a rhombus, or $u=1$.
In this last case, we compute the derivative $\partial \psi/\partial t$ and we get for $u=1$
$$(t+1)^2 \frac{\partial \psi}{\partial t}(t,1)= 2t^3+4t^2+2t+2H^2$$
that is positive, too. Therefore, when $u=1$, the only new case is $t=1$ that corresponds to a rectangle.
We conclude by comparing $W/A$ for the rhombus (whose value is $(1+H^2)/6$) and for the rectangle (whose value is $(1+H^2)/3$). This gives the desired result.
\end{proof}

\medskip\noindent
{\bf 4th step: Conclusion}
We come back to the quantity $\mathcal{E}$ defined in \eqref{quantD} that we write with the three points
(we keep the notations $t,u$ introduced previously to denote the interior point)
$N_0=(0,H);N_1=(t,uH);N_2=(1,0)$ and we use the angles $\theta_1,\theta_2$, in particular to write
$$t=N_0N_1 \cos\theta_1,\  (1-u)H=N_0N_1 \sin\theta_1, \ 1-t=N_1N_2 \cos\theta_2, \ uH=
N_1N_2 \sin\theta_2.$$
We recall that $\theta_1\leq \theta\leq \theta_2$ by convexity.
The quantity $A\mathcal{E}$ is the sum of four terms that we denote $\mathcal{E}_1,\mathcal{E}_2,
\mathcal{E}_3,\mathcal{E}_4$:
$$\mathcal{E}_1=\frac{8A}{P}\left(\frac{1-\sin(2\theta_1)}{N_0N_1}\,+ \frac{1-\sin(2\theta_2)}{N_1N_2}\right).$$
For the second and third terms, it is convenient to introduce $\delta:=(\theta_2-\theta_1)/2$ and
$\mu:=(\theta_2+\theta_1)/2$. With these notations, we can write
$$\mathcal{E}_2=\frac{128A}{P^2}\,\sin^2\delta (\cos\mu-\sin\mu)^2$$
$$\mathcal{E}_3=\frac{64}{P}\,(1-H)\sin\delta (\cos\mu-\sin\mu)$$
and the last term is
$$\mathcal{E}_4=\frac{2A}{3W}\left(1+2t+H^2(1+2u)-4H(t+u)\right).$$
It is clear that $\mathcal{E}_1, \mathcal{E}_2$ are always positive. This is also the case
for $\mathcal{E}_3$ as soon as $\mu\leq \pi/4$.
Now, since $H$ has been supposed to be less than $1/2$ (we have chosen $H=1/4$ previously, the minimum
of $\mathcal{E}_4$ is certainly obtained by taking $t=0$ and $u=1$ (since $2-4H\geq 0$ and $2H^2-4H\leq 0$).
This means 
$$\mathcal{E}_4 \geq \frac{2A}{3W}\left(1+2 H^2-4H\right) >0 \quad \mbox{when } H<1-\frac{1}{\sqrt{2}}.$$
In conclusion $\mathcal{E}>0$ for $H=1/4$ and $\mu \leq \pi/4$.

It remains to consider the case $\mu \geq 1/4$ In that case, it is easier to work with the trace of the Hessian
matrix, namely
$$\mathcal{T}= \frac{\partial^2 F}{\partial x_1^2} +  \frac{\partial^2 F}{\partial y_1^2} \,.$$
It can be written
$$\mathcal{T}=\frac{8A}{P}\left(\frac{1}{N_0N_1}+\frac{1}{N_1N_2}\right)+\frac{128 A}{P^2}\sin^2\delta
+\frac{64}{P}\sin\delta (H\sin\mu +\cos\mu)-\frac{8A H}{3W} (t+u).$$
Here we have also four terms $\mathcal{T}_1,\mathcal{T}_2,\mathcal{T}_3,\mathcal{T}_4$
that we estimate separately. 
Taking into account $\mu \geq \pi/4$, we infer 
$$\frac{\pi}{4} \leq \frac{\theta_1+\theta_2}{2} \leq \frac{\theta+\theta_2}{2}$$
therefore $\theta_2\geq \frac{\pi}{2} -\theta$. This implies that the point $(t,uH)$ lies into the
triangle defined by the three points $(1,0);(1,H);(1-H\tan\theta,H)$. In particular,
$N_0N_1 \leq \sqrt{1+H^2}$, while $N_1N_2 \leq H/\cos\theta$. Using now $8A/P\geq 4H/\sqrt{1+H^2}$
coming from Proposition \ref{propesti}, we see that
$$\mathcal{T}_1=\frac{8A}{P}\left(\frac{1}{N_0N_1}+\frac{1}{N_1N_2}\right) \geq \frac{4(1+H)}{1+H^2}$$
where we used $\cos\theta=1/\sqrt{1+H^2}$.

We consider the second term. From $\theta_2\geq \pi/2 -\theta$, we deduce $\delta \geq \pi/4-\theta$, thus
$\sin\delta \geq (\cos\theta-\sin\theta)/\sqrt{2}$.  We use $A/P^2\geq H/(8(1+H^2)$ coming from Proposition \ref{propesti} to infer
$$\mathcal{T}_2=\frac{128 A}{P^2}\sin^2\delta \geq \frac{8H(1-H)^2}{(1+H^2)^2}.$$

For the third term, we use the identity $H\sin\mu +\cos\mu=\cos(\mu-\theta)/\cos\theta$ with the inequality
$\mu-\theta\leq \pi/4-\theta/2$ to have $\cos(\mu-\theta)\geq (\cos\frac{\theta}{2} +\sin\frac{\theta}{2})/\sqrt{2}
\geq 1/\sqrt{2}$, and we also use $1/P\geq 1/4(1+H)$ and $\sin\delta \geq (\cos\theta-\sin\theta)/\sqrt{2}$ to infer
$$\mathcal{T}_3=\frac{64}{P}\sin\delta (H\sin\mu +\cos\mu) \geq \frac{8(1-H)}{1+H}.$$

At last, we use $8A/3W \leq 16/(1+H^2)$ coming from Proposition \ref{propesti} and 
$t+u\leq 2$ to deduce that the fourth term is estimated by
$$\mathcal{T}_4=-\frac{8A H}{3W} (t+u)\geq -\frac{32 H}{1+H^2}.$$
Summing these four terms, we finally get that the trace of the Hessian matrix is estimated from below by
$$\mathcal{T}\geq \frac{4}{(1+H)(1+H^2)^2}\left(3-6H-4H^2-12H^3-3H^4-2H^5\right)$$
therefore $\mathcal{T}>0$ for $H=1/4$ yielding the desires contradiction.

We are led to the only following possibilities for the optimal domain
\begin{enumerate}
\item it is the rhombus
\item it is the rectangle
\item it has another vertex on the horizontal line $y=H$
\item or/and it has another vertex on the vertical line $x=1$.
\end{enumerate}
We are going to prove that we are necessarily in the two first cases.
 Let us assume that the optimal domain has two
vertices  $M_1=(t,H)$ and $M_2=(1,uH)$ with $(t,u)\in [0,1]^2$. In that case, we have
$$A=2H(1+t+u-tu), \quad P=4\left(t+uH+\sqrt{(1-t)^2+H^2(1-u)^2}\right)$$
$$W=\frac{H}{3}\left(4u+(1-u)(1+t+t^2+t^3)\right)+
\frac{H^3}{3}\left(4t+(1-t)(1+u+u^2+u^3)\right)$$
Let us introduce the quantity $G=P^2A-96W$ and we want to prove that $G\leq 0$ for any
$(t,u)\in [0,1]^2$ and any $H\in (0,1]$. For that purpose, we compute the first and second derivative of $G$
with respect to $t,u$. We denote by $d$ the distance between $M_1$ and $M_2$, namely
$d=\sqrt{(1-t)^2+H^2(1-u)^2}$
$$
\frac{\partial G}{\partial t}=8AP\left(1-\frac{1-t}{d}\right)+2H(1-u)P^2-32H(1-u)\left(1+2t+3t^2+H^2(3+2u+u^2)\right).
$$
$$
\frac{\partial G}{\partial u}=8APH\left(1-\frac{H(1-u)}{d}\right)+2H(1-t)P^2-32H(1-t)\left(3+2t+t^2+H^2(1+2u+3u^2)\right).
$$
The second derivative of $G$ in $u$ can be written as the sum of four terms as
$$\frac{\partial^2 G}{\partial u^2}=32H(\mathcal{E}_1+\mathcal{E}_2+\mathcal{E}_3+\mathcal{E}_4)$$
with
$$\mathcal{E}_1=(1+t+u-tu)H^2 \left(1-\frac{H(1-u)}{d}\right)^2$$
$$\mathcal{E}_2=(1+t+u-tu)H^2(t+uH+d)(1-t)^2/d^3$$
$$\mathcal{E}_3=2(1-t)H(t+uH+d)\left(1-\frac{H(1-u)}{d}\right)$$
$$\mathcal{E}_4=-H^2(1+3u)(1-t).$$
We follow now the same strategy than previously: thanks to the affine change of variable
$x^\prime=x, y^\prime=y/(4H)$ we are led to examine the situation of the second derivatives where $H=1/4$.
Let us look at $\mathcal{E}_3+\mathcal{E}_4$. Using $t+uH+d\geq \sqrt{1+H^2}$ (estimate of the perimeter
by the perimeter of the rhombus) and $d\leq \sqrt{1+H^2}$, we get
$$\mathcal{E}_3+\mathcal{E}_4 \geq 2H(1-t)\left(\sqrt{1+H^2} -H(1-u) -H^2(1+3u)\right)$$
now, it is clear that the right-hand side is positive for $H=1/4$ for all $t,u$. Since the two first terms 
$\mathcal{E}_1,\mathcal{E}_2$ are positive, we infer that the second derivative
$\frac{\partial^2 G}{\partial u^2}$ is positive: that implies that $G$ cannot have a maximizer for $0<u<1$.
Thus, for a maximizer, we necessarily have $u=0$ or $u=1$. The case $u=1$ corresponds to a rectangle.
It remains to look at the case $u=0$. In that case, we are going to prove directly that $G(t,0)<0$ for $0<t<1$
it will imply that in that case, the maximizing domain is a rhombus and a simple comparison of
$G$ in these two cases provides $G=0$ for the rhombus and $G=-64H(H-1)^2$ that shows that
the maximizing domain is any rhombus (including the square that we recover in the case of a rectangle with
$H=1$).

Let us prove that $G(t,0) <0$ when $0<t<1$. For that purpose, we use first the convexity of the
function $t\mapsto \sqrt{(1-t)^2+H^2}$ to claim that
$$
\sqrt{(1-t)^2+H^2}\leq \sqrt{1+H^2} +t(H-\sqrt{1+H^2}).
$$
Then, we can estimate the perimeter squared by
$$P^2\leq 16\left(1+H^2-2t+2t^2+2t(\sqrt{1+H^2} +t(H-\sqrt{1+H^2})\right).$$
This allows to estimate $G=P^2A-96W$ by
$$G(t,0)\leq t(a_1t^2+a_2t+a_3)$$
with
$$a_1=1+2H-2\sqrt{1+H^2},\;a_2=2H-1,\;a_3=2\sqrt{1+H^2}-2-2H^2.$$
We have $a_1\geq 0 \Leftrightarrow H\geq 3/4$, $a_2\geq 0 \Leftrightarrow H\geq 1/2$ and $a_3<0$.
Then, $P_H(t)=a_1t^2+a_2t+a_3$ is clearly negative when $H\leq 1/2$. When $H>3/4$, $P_H(t)$ is increasing
on $[0,1]$ since its minimum is attained at a negative $-a_2/2a_1$. But $P_H(1)=-2(H-1)^2\leq 0$;
In the last case, $1/2<H\leq 3/4$, $P_H$ has a maximum at $t_1=-a_2/2a_1$ that is positive. 
But when $H\geq (-3+2\sqrt{21})/10$ this maximum point being outside $[0,1]$, $P_H$ is still increasing, then negative on $[0,1]$.
At last, when $H\in [1/2,  (-3+2\sqrt{21})/10]$, $t_1\in [0,1]$ and the maximum of $P_H$ is 
$$P_H(t_1)=-\frac{a_2^2}{4a_1} + a_3$$
and it is straightforward to check that this number is negative.
Therefore, we have proved that $P_H(t)=a_1t^2+a_2t+a_3$ is always negative when $0<t<1$.
\qed

\section{Behavior near $(0,0)$ and $(1,1)$}\label{sec-pentes}
In this section we investigate the behavior of the diagram near the two ``extremal'' points, on the left and on the right of the diagram. These points are: the origin, associated to thin domains collapsing to a segment, and the point $(1,1)$ associated to the ball.

\subsection*{Behavior near the origin $(0,0)$}

We will compute the minimal slope $\gamma^-_O$ and the maximal slope $\gamma^+_O$ at the origin, defined as
\begin{align*}
\gamma^-_O & := \liminf\left\{ \frac{y_\e}{x_\e}\ :\ (x_\e, y_\e) \in \mathcal D\,,\ (x_\e, y_\e) \to (0,0) \right\}, \quad 
\\
\gamma^+_O & := \limsup\left\{ \frac{y_\e}{x_\e}\ :\ (x_\e, y_\e) \in \mathcal D\,,\ (x_\e, y_\e) \to (0,0) \right\}. 
\end{align*}
To this aim, let us first rewrite in a more tractable way the ratios $y/x$ for $(x,y)\in \mathcal D$. By definition, the pair $(x,y)\in \mathcal D$ is associated to an admissible shape $\Omega \in \mathcal A$, which in turn is characterized by its intersection with the first quadrant, say $\Omega^+:=\Omega \cap \{x\geq 0 , y\geq 0\}$. This set can be described as
\begin{equation}\label{h}
\Omega^+=\{0 \leq x \leq x_0, \quad 0 \leq y\leq h(x)\},
\end{equation}
for some $x_0>0$ and $h$ concave decreasing. 
Therefore we can rewrite
$$
\frac{y}{x} = \frac{y(\Omega)}{x(\Omega)} = \frac{\pi^2}{2} \frac{\int_0^{x_0} (x^2 h(x) + h^3(x)/3)\, \mathrm{d}x}{\left[\int_0^{x_0} \sqrt{1+(h'(x))^2} \, \mathrm{d}x + h(x_0)\right]^2 \int_0^{x_0} h(x)\, \mathrm{d}x}.
$$

\begin{proposition}
The minimal and maximal slopes at the origin are
$$
\gamma^-_O=\frac{\pi^2}{12}, \quad \gamma_O^+= \frac{\pi^2}{6}.
$$
Moreover, they are attained, e.g., by sequences of thin rhombi and thin rectangles, respectively.
\end{proposition}

\begin{proof}
In view of Corollary \ref{corAH}, we immediately get $\gamma^-_O\geq \pi^2/12$. The equality sign is obtained e.g. by taking a sequence of thin rhombi shrinking to a segment.

Let us study the upper bound. Let $\Omega_\e$ be a sequence satisfying the $\limsup$ characterized by some $h_\e$ according to \eqref{h}. Without loss of generality, up to a rotation of $\pi/2$, we may assume that $h_\e$ is defined in the interval $[0,1]$ and that $h_\e(0)\to 0$. Thus $h_\e$ is of the form $h_\e(x)= c_\e f_\e(x)$, with $c_\e\to 0$ and $f_\e$ in the following class:
$$
\mathcal F:=\{ f:[0,1] \to \mathbb R\ :\ f (0)=1, \ f \ \hbox{concave, decreasing, non negative}\}.
$$
Then
 \begin{align*}
 \frac{y(\Omega_\e)}{x(\Omega_\e)} & = \frac{\pi^2}{2} \frac{\int_0^{1} (x^2 c_\e f_\e(x) + c_\e^3 f_\e^3(x)/3)\, \mathrm{d}x}{\left[\int_0^{1} \sqrt{1+(c_\e f_\e'(x))^2} \, \mathrm{d}x + c_\e f_\e(1)\right]^2 \int_0^{1} c_\e f_\e(x)\, \mathrm{d}x}
\\ & = 
 \frac{\pi^2}{2} \frac{ \int_0^{1}x^2 f_\e(x) \, \mathrm{d}x + o(1) }{\int_0^{1} f_\e(x)\, \mathrm{d}x + o(1)}.
\end{align*}

This ratio is clearly bounded above by $\pi^2/2$. To get a finer estimate, we show that for every $f\in \mathcal F$ 
\begin{equation}\label{enough}
\int_0^1 (x^2-1/3) f(x)\, \mathrm{d}x \leq 0.
\end{equation}
This fact easily comes by splitting the interval of integration as $[0,1]=[0,1/\sqrt{3}] \cup [1/\sqrt{3},1]$. Using the monotonicity and positivity of $f$, we infer that the integrand is negative in $[0,1/\sqrt{3}]$ and positive in $[0,1/\sqrt{3}]$. In both cases the integrand is bounded above by the function $f(1/\sqrt{3}) (x^2-1/3)$, which has zero average in $[0,1]$. This gives \eqref{enough}, in particular the estimate
$$
\gamma^-_O \leq \frac{\pi^2}{6}.
$$
In order to prove that $\gamma^-_O=\pi^2/6$ it is enough to exhibit a sequence of admissible shapes for which the ratio $y_\e/x_\e \to \pi^2/6$. This is the case of thin rectangles, in which $f_\e\equiv 1$ for every $\e$. 
\end{proof}

\begin{remark} Notice that the highest slope associated to thin domains coincides with the slope of the curve associated to the stadiums. Let $\Omega_L$ be the stadium associated to two half disks of centers $\pm L$ and radius $1$, then as $L\to +\infty$, $(x(\Omega_L), y(\Omega_L))\to (0,0)$, since
$$
A(\Omega_L)= \pi + 4L\,,\quad P(\Omega_L)= 2\pi + 4 L\,,\quad W(\Omega_L)= \frac{\pi}{2} +  4 L +  \pi L^2 +  \frac43 L^3.
$$
The slope at the origin is 
$$
\lim_{L\to +\infty} \frac{y(\Omega_L)}{x(\Omega_L)} =\frac{\pi^2}{6}.
$$
\end{remark}

\subsection*{Behavior near $(1,1)$}
As in the previous section, we aim to compute the minimal slope $\gamma^-_{\mathbb D}$ and the maximal slope $\gamma^+_{\mathbb D}$ at the point $(1,1)$, defined as the left derivatives of $L^+$ and $L^-$ in $x=1$, respectively. We will prove that the two following limits exist
$$
\gamma^-_{\mathbb D}:=\lim_{x\to 1^-} \frac{L^+(x)- 1}{x-1}, \quad \gamma^+_{\mathbb D}:=\lim_{x\to 1^-} \frac{L^-(x)- 1}{x-1},
$$
and we will compute their value.

\begin{proposition}\label{propslope} The minimal and maximal slopes at $(1,1)$ are
$$
\gamma^-_{\mathbb D}= \frac34, \quad \gamma^+_{\mathbb D}=+ \infty.
$$
\end{proposition}
\begin{proof}
The proof is divided into several steps.

{\it Step 1: the sequence of regular polygons.} Let $\widehat{\Omega}_n$ be the regular polygons with $2n$ sides, $n\geq 2$, with outer radius 1. Thus we have
$$
A(\widehat{\Omega}_n)= 2 n \sin (\pi/2n) \cos(\pi/2n)  , \quad P(\widehat{\Omega}_n)=4 n \sin (\pi/2n),
$$
$$
W(\widehat{\Omega}_n) = n \cos^4(\pi/2n) \left(\tan(\pi/2n) + \frac{\tan^3(\pi/2n)}{3}  \right).
$$
So that 
$$
\widehat{x}_n:= x(\widehat{\Omega}_n)
= \frac{  \sin (\pi/2n)  }{(\pi/2n)  \cos(\pi/2n)  (1 + \tan^2(\pi/2n)/3)}
$$
and
$$
\widehat{y}_n:= y(\widehat{\Omega}_n) = \frac{ (\pi/2n) \cos(\pi/2n)}{\sin(\pi/2n)}.
$$
Along this sequence, it is immediate to check that
$$
\lim_{n \to \infty }\frac{\widehat{y}_n-1}{\widehat{x}_n-1} = +\infty.
$$

{\it Step 2: the maximal slope.} Let $x_n\to 1^-$ be a sequence realizing
$$
\lim_{n\to \infty} \frac{1-L^-(x_n)}{1-x_n} = \liminf_{x \to 1} \frac{1-L^-(x)}{1-x}.
$$ 
We can extract an increasing subsequence $x_{n_k}$ such that $ \widehat{x}_{n_k} \leq x_{n_k} \leq \widehat{x}_{n_{k}+1}$.  On the one hand, we clearly have $1 - x_n \leq 1- \widehat{x}_{n_k} $. On the other hand, since $L^-$ is increasing and defined as a minimum, we have $1-L^-(x_n) \geq 1-\widehat{y}_{n_{k}+1}$. Using these estimates together with Step 1, we conclude that
$$
\lim_{n\to \infty} \frac{1-L^-(x_n)}{1-x_n} \geq \lim_{k \to \infty}\frac{1-\widehat{y}_{n_{k}+1}}{1- \widehat{x}_{n_k}} = + \infty.
$$

{\it Step 3: the minimal slope.}  In this part we use the representation of shapes via polar coordinates: $\Omega=\{(\rho \cos \theta, \rho \sin \theta) \ :\ \theta \in[0,2\pi]\,,\ \rho \in [0,\rho_{max}(\theta)]\}$, for a suitable function $\rho_{max}$.

Let $x_\e\to 1^-$ as $\e\to 0$ be a sequence realizing
$$
\lim_{\e \to 0}  \frac{1-L^+(x_\e)}{1-x_\e} = \liminf_{x \to 1} \frac{1-L^+(x)}{1-x}.
$$ 
Let $\Omega_\e$ be the associated shapes in $\mathcal A$. So that
\begin{equation}\label{ratioe}
 \frac{1-L^+(x_\e)}{1-x_\e} =  \frac{1-y(\Omega_\e)}{1-x(\Omega_\e)}.
\end{equation}
 Using the representation in polar coordinates, the shapes are characterized by a function $r_\e(\theta)$ as follows:
$$
\Omega_\e=\{(\rho \cos \theta, \rho \sin \theta) \ :\ \theta \in[0,2\pi]\,,\ \rho \in [0,r_{\e}(\theta)]\}.
$$
Since $x_\e\to 1^-$, without loss of generality, we may assume that $\Omega_\e$ converges to the disk of unit radius $\mathbb D$, with respect to the complementary Hausdorff distance. Therefore we infer that $r_\e \to 1$ uniformly and it can be written as  $r_\e(x)= 1 + c_\e f_\e(x)$, for some $c_\e \to 0$ and $f_\e\in C^0([0,2\pi])$ with $\|f_\e\|_{\infty}= 1$. 
Let us write the Taylor developments in $\e=0$ of $A,P,W$ computed at $\Omega_\e$:
\begin{align*}
A(\Omega_\e)& = \int_0^{2\pi} \frac{(1+ c_\e f_\e)^2}{2} = \pi + c_\e \int_0^{2\pi} f_\e  + \frac{c_\e^2}{2} \int_0^{2\pi}f_\e^2
\\
& = A(\D) + 4 c_\e \int_0^{\pi/2} f_\e + 2 c_\e^2 \int_0^{\pi/2} f_\e^2 ,
\\
P(\Omega_\e) & = \int_0^{2\pi} \sqrt{r_\e^2 + \dot{r_\e}^2 } = \int_0^{2\pi}\sqrt{ 1 + 2 c_\e f_\e + c_\e^2 (f_\e^2 + \dot{f}_\e^2)}
\\ & = \int_0^{2\pi} \left[1 + c_\e f_\e + \frac{c_\e^2}{2}(f_\e^2 + \dot{f}_\e^2) - \frac18(2c_\e f_\e)^2\right]+ o(c_\e^2) 
\\ & = P(\D) + 4 c_\e \int_0^{\pi/2} f_\e + 2 c_\e^2 \int_0^{\pi/2} \dot{f}_\e^2  + o (c_\e^2),
\\
W(\Omega_\e) & = \int_0^{2\pi} \frac{(1+c_\e f_\e)^4}{4} = \frac14 \int_0^{2\pi} [1 + 4 c_\e f_\e + 6 c_\e^2 f_\e^2 ] + o(c_\e^2)
\\ & = W(\D) + 4 c_\e \int_0^{\pi/2} f_\e + 6 c_\e^2 \int_0^{\pi/2} f_\e^2 + o (c_\e^2). 
\end{align*}
Note that here we have used the double symmetry to replace the integrals over $[0,2\pi]$ as integrals over $[0,\pi/2]$. Inserting these expressions in \eqref{ratioe}, and denoting by $g_\e(x):= f_\e(x) - \fint f_\e $, we obtain 
$$
 \frac{1-y(\Omega_\e)}{1-x(\Omega_\e)} =\frac14 \frac{\left( \int_0^{\pi/2} f_\e \right)^2 + 2 \pi \int_0^{\pi/2}(\dot{f}_\e^2 - f_\e^2)}{2\pi \int_0^{\pi/2}f_\e^2 - \left( \int_0^{\pi/2} f_\e \right)^2} + o(1)
 = \frac14 \left( \frac{ \int_0^{\pi/2} \dot{g}_\e^2}{\int_0^{\pi/2}g_\e^2 } - 1 \right) + o(1). 
$$
The last expression can be bounded from below with the first non trivial Neumann eigenvalue $\mu_1((0,\pi/2))$ of the interval $(0,\pi/2)$:
$$
 \frac{ \int_0^{\pi/2} \dot{g}_\e^2}{\int_0^{\pi/2}g_\e^2 } \geq \inf_{\int_0^{\pi/2} g = 0\,,\ g\not\equiv 0 }  \frac{ \int_0^{\pi/2} \dot{g}^2}{\int_0^{\pi/2}g^2 } =:\mu_1((0,\pi/2))=4.
 $$
Thus, passing to the limit $\e\to 0$, we obtain 
$$
\gamma_\D^-\geq \frac34.
$$
If we are able to find a sequence of shapes converging to the disk with slope $3/4$ we are done. A sequence with this property is the sequence of ellipses $E_\e$ with semi-axes $1$ and $1+\e$. For these shapes we have
\begin{align*}
A(E_\e)&= \pi (1+\e),  \quad W(\Omega_\e)= \frac{\pi}{4}[(1+\e)^3 + (1+\e)], 
\\
P(E_\e)&=\int_0^{2 \pi} \sqrt{1  - 2 \e \sin^2(\theta) + \e^2 \sin^2(\theta)}\, \mathrm{d}\theta.
\end{align*}
A direct computation gives
$$
1- x(E_\e)=\frac{\e^2}{2} + o(\e^2),\quad 1-y(E_\e)=\frac{3\e^2}8 + o(\e^2),
$$
therefore, in the limit as $\e \to 0$, we have the desired result:
$$
\lim_{\e \to 0^+} \frac{1-y(E_\e)}{1-x(E_\e)} = \frac 34.
$$
\end{proof}

\section{Numerics}
In this section, we perform simple numerical methods to try to identify the optimal domains on the upper and lower boundaries
of the Blaschke-Santal\'o diagram. We choose a different methods for each boundary, since we already know that the optimal
domains are "smooth" (at least $C^{1,1}$) on the upper boundary while they are polygonal on the lower boundary.
We also compare the best domains we get numerically with the candidates as they appear for example in Reference \cite{BBO}.
\subsection{The upper boundary $L^+$}
For a given abscissa $x_0\in (0,1)$, the problem consists in {\it minimizing} the perimeter among convex sets $\Omega$ that
satisfy for example $A(\Omega)=\pi$ and $W(\Omega)=\pi/2x_0$. Since we know that the optimal domain is smooth, we choose
to represent the convex domain by its support function $p(\theta)$ as suggested for example in \cite{ABo} and \cite{Bo}.
Then we have the choice
\begin{enumerate}
\item either to decompose the support function in Fourier series, and the unknown are the Fourier coefficients.
\item or to discretize the support function by looking for its value $p_i$ on a discrete grid $\theta_i$.
\end{enumerate}
These two methods are discussed and implemented (in particular for spectral problems) in  \cite{ABo} and \cite{Bo}.
Here we choose the Fourier series decomposition. Due to the symmetries, the Fourier series of $p$ writes
$$p(\theta)=\sum_{k=0}^N a_{2k} \cos 2k\theta.$$
The reason of our choice is the following. First,  the geometric quantities are either exactly calculable (perimeter, area)
in terms of the Fourier coefficients:
$$P(\Omega)=2\pi a_0,\quad A(\Omega)=\frac{\pi}{2}\left(2a_0^2+\sum_{k=1}^N (1-4k^2)a_{2k}^2\right)$$
or can be computed with a very good accuracy (Simpson rule for example) for the moment of inertia:
$$W(\Omega)=\frac{1}{12} \int_0^{2\pi} \left[3p^4-6p^2{p^\prime}^2- {p^\prime}^4\right] \mathrm{d}\theta.$$
Then, the most important point is that this choice allows us to consider a small-medium scale optimization problem. Indeed,
we have chosen to work with 16 Fourier coefficients (where a discretization method would require several hundred unknowns
for the same accuracy). Of course, this choice has drawbacks, since it is impossible to capture shapes with segments by considering
a truncated Fourier series: our support function being $C^\infty$, the shape is strictly convex (see \cite{Sch}).
As we will see, the stadium (or a stadium-like shape) seems to be the optimal domain for a range of values $x_0\in [0,0.72]$ and it
is slightly difficult to capture: we probably need to work with more Fourier coefficients here.

In \cite{BBO}, for the same problem, some numerical results are presented and the authors seem to identify two families of optimal
domains for this problem: stadium-like and ellipse-like. We recover a similar result here, but let us mention two points:
\begin{itemize}
\item it is not because the shape looks like a stadium and the numerical values are very similar that it is a stadium! We have a famous counter-example
in \cite{HO}
\item The ellipse {\bf cannot} be an optimal domain, as it is proved in the following proposition.
\end{itemize}
\begin{proposition}
An ellipse (different from a disk) cannot be a solution of the problem
$$\min \{P(\Omega), \mbox{ with } A(\Omega)=A_0,\,W(\Omega)=W_0\}.$$
\end{proposition}
\begin{proof}
Let us consider an ellipse of semi-axes $a>b$ whose parametrization is $x(t)=a\cos t, y(t)=b\sin t, t\in [0,2\pi]$.
Since it is a convex domain $C^2_+$ (in the sense that the curvature is bounded from below by a positive constant), we can perform
any variation of its boundary. The shape derivative techniques, see \cite[Chapter 5]{HP} shows that there exist two Lagrange multipliers $\lambda_1,\lambda_2$
such that
$$\forall V\in C^0(\mathbb{R}^2,\mathbb{R}^2),\ \int_{\partial\Omega} \mathcal{C} V.\nu =\lambda_1 \int_{\partial\Omega} V.\nu + \int_{\partial\Omega} (x^2+y^2) V.\nu$$
where $\mathcal{C}$ is the curvature and $\nu$ the exterior normal vector. Since this is true for any $V$, this implies
\begin{equation}\label{ellipse}
\mathcal{C}=\lambda_1 + \lambda_2 (x^2+y^2).
\end{equation}
Using the parametrization of the boundary, \eqref{ellipse} can be written
\begin{equation}\label{elli2}
\frac{ab}{\left(a^2\sin^2 t+b^2\cos^2 t\right)^{3/2}} =\lambda_1 + \lambda_2 (a^2\cos^2 t+b^2\sin^2 t).
\end{equation}
Differentiating \eqref{elli2} yields
$$\frac{3ab(b^2-a^2) \sin 2t}{2\left(a^2\sin^2 t+b^2\cos^2 t\right)^{5/2}} = \lambda_2 (b^2-a^2) \sin 2t$$
and we see that the previous equality can hold for any $t$ if and only if $a=b$, that proves the claim.
\end{proof}
For any given abscissa $x_0$ there is only one (up to scaling) stadium and ellipse satisfying $A^2/(2\pi W)=x_0$. In the following table, we represent for different
values of $x_0$, the corresponding value of $y=4\pi A/P^2$ for the ellipse (2nd column), the stadium (3rd column) and the best domain we got with our numerical
procedure where we use the routine {\it fmincon} of Matlab, providing the gradients of the objective function and the constraints (with 16 Fourier coefficients).
The convexity constraint is linear in the unknowns $a_{2k}$ and is simply written as $(p+p^{\prime\prime})(\theta_j)\geq 0$ for the same discretization of $(0,\pi/2)$ used
for computing the moment of inertia.
We recall that minimizing the perimeter is equivalent here to maximize $y$, so we are looking for the largest value of $y$.
$$\begin{array}{c|c|c|c}
x_0 & \mbox{ellipse} & \mbox{stadium} & \mbox{best numerical} \\\hline
0.1 & 0.1225 & 0.2522 & ? \\
0.2 & 0.2413 & 0.2839 & 0.2420 \\
0.3 & 0.3551 & 0.3998 & 0.3870 \\
0.4 & 0.4634 & 0.5036 & 0.4861 \\
0.5 & 0.5660 & 0.5977 & 0.5841 \\
0.6 & 0.6630 & 0.6844 & 0.6777 \\
0.7 & 0.7546 & 0.7654 & 0.7641 \\
0.75 & 0.7985 & 0.8043 & 0.8050 \\
0.8 & 0.8411 & 0.8425 & 0.8456 \\
0.85 & 0.8825 & 0.8804 & 0.8851 \\
0.9 & 0.9228 & 0.9184 & 0.9241 \\
0.95 & 0.9619 & 0.9572 & 0.9623 \\
\end{array}$$
Our observations are the following:
\begin{itemize}
\item For $x_0\leq 0.72$, we are not able to get numerically a better  domain (with 16 Fourier coefficients) than the stadium. After $0.72$ we are always able to get
a domain which is better than the stadium, but also better than the ellipse.
\item When $x_0$ approaches $1$, the domain we get numerically gives a value of $y$ closer and closer to the one given by the ellipse, confirming Proposition \ref{propslope}
and the fact that ellipses provide the slope of the curve $L^+$ when we approach the point $(1,1)$.
\end{itemize}
Here we plot two possible optimal domains for the values $x_0=0.8$ and $0.9$.

\begin{center}
\begin{figure}[h]
\includegraphics[scale=0.15]{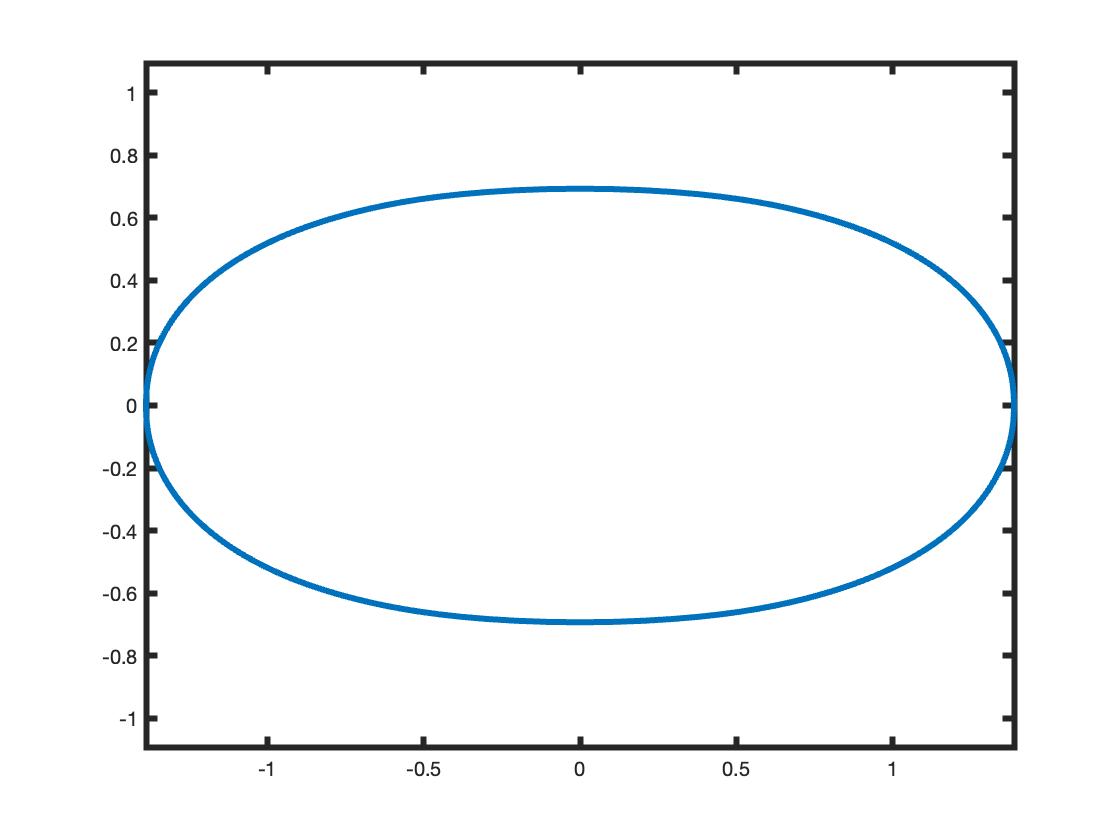}
\includegraphics[scale=0.15]{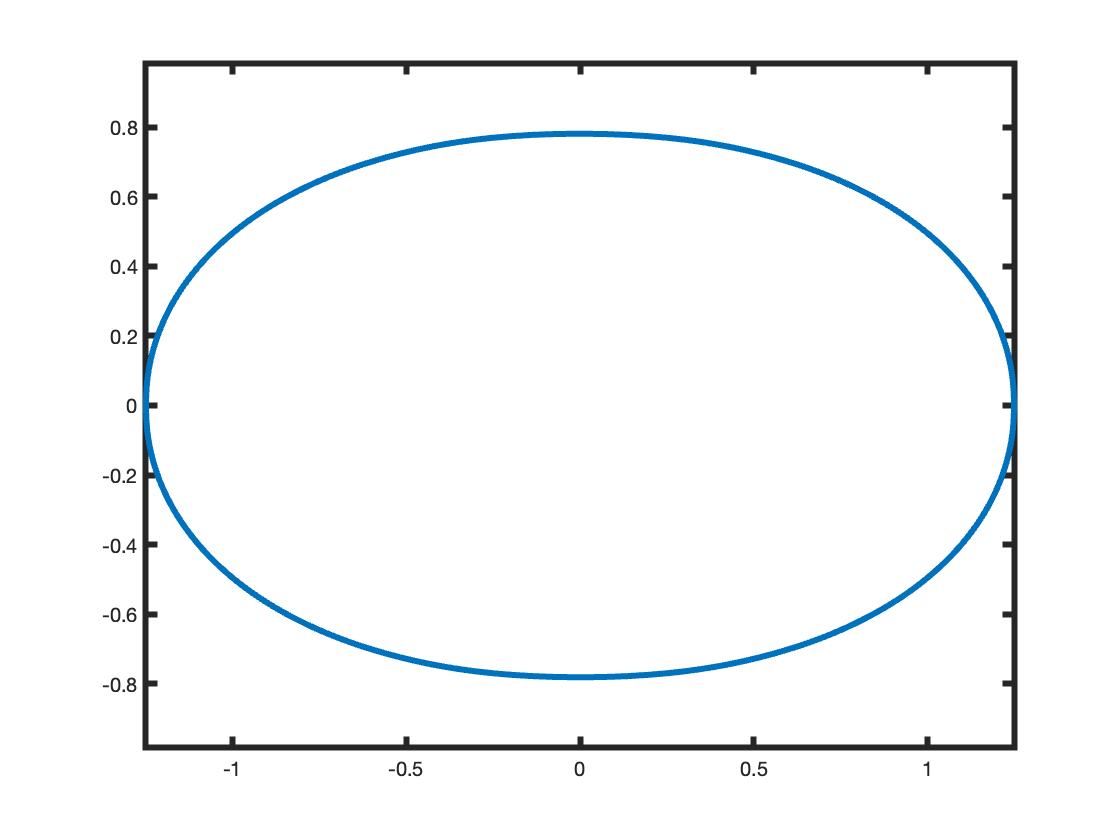}
\caption{A possible optimal domain for $x_0=0.8$ (left), for $x_0=0.9$ (right).}
\end{figure}
\end{center}
\subsection{The lower boundary $L^-$}
We already know that the optimal domains are polygons. We also know that for $x_0\leq 3/\pi\simeq 0.9549$ the optimal domain is the unique rhombus (up to scaling)
satisfying $A^2/(2\pi W)=x_0$. It remains to find the optimal (polygonal) domains for $\frac{3}{\pi} <x<1$. For that purpose, we consider $k+1$ unknown vertices
$M_i=(x_i,y_i), i=0,\ldots k$ where we assume $x_0=0$ and the last vertex is known and taken as $M_{k+1}=(1,0)$.  We have $2k+1$ unknown, the quantity under consideration are
$$P(\Omega)=4\sum_{i=0}^{k} \sqrt{(x_{i+1}-x_i)^2+(y_{i+1}-y_i)^2},\ A(\Omega)=2\left(y_k+\sum_{i=1}^k x_i(y_{i-1}-y_{i+1})\right)$$
\begin{align*}
W(\Omega)=& \frac{1}{3}\sum_{i=0}^k (x_{i+1}-x_i)(y_i^3+y_i^2y_{i+1}+y_iy_{i+1}^2+y_{i+1}^3) 
\\ & - \frac13 \sum_{i=0}^k (x_i^2+x_ix_{i+1}+x_{i+1}^2)(x_iy_{i+1}-x_{i+1}y_i).
\end{align*}
In the paper \cite{BBO}, the optimal domains looks like some "regular" octagon, therefore we have compared in the array below:
\begin{itemize}
\item this family or regular octagons (whose vertices in the first quadrant are $(0,1);(x_1,x_1);(1,0)$ where $x_1$ is chosen in such a way that $A^2/(2\pi W)=x_0$
\item the best hexagon we get in the family of hexagons whose vertices in the first quadrant are $(0,H);(x_1,H);(1,0)$,
\item the best polygon we get with our numerical procedure (we still use  the routine {\it fmincon} of Matlab with a maximal value of $k$ equal to 5.
\end{itemize}
The convexity constraint is non-linear in the unknowns and  writes, for any internal vertex $1\leq i\leq k$:
$$(x_{i+1}-x_i)y_{i-1} +(x_i-x_{i-1})y_{i+1}-(x_{i+1}-x_{i-1})y_i \leq 0$$
Since here we want to maximize the perimeter, this corresponds to the lower value of $y=4\pi A/P^2$.
$$\begin{array}{c|c|c|c}
x_0 & \mbox{"regular" octagon} & \mbox{best hexagon} & \mbox{best numerical} \\\hline
0.956 & 0.78893 & 0.78892 & 0.78892 \\
0.96 & 0.8023 & 0.8021 & 0.8021 \\
0.97 & 0.8376 & 0.8353 & 0.8353 \\
0.98 & 0.8794 & 0.8682 & 0.8682 \\
0.985 & 0.8946 & 0.8843 & 0.8843 \\
0.99 & 0.9150 & 0.8998 & 0.8998 \\
0.995 & 0.9362 & none & 0.9290 \\
\end{array}$$
Comments: hexagons exist up to $x_0=9\sqrt{3}/(5\pi)\simeq 0.9924$ that is the value for the regular hexagon.
By the way, the authors do not know whether the maximization of $A^2/W$ among polygons with a given number of sides is always given by the regular polygon.
When hexagons exist, i.e. when $x_0<0.9924$, the best hexagons seem to be the optimal sets and, in particular are better than the regular octagons.
After $0.9924$ and before $12(\sqrt{2}+1)/(\pi(5+3\sqrt{2}))\simeq 0.9977$ that is the value of $x$ for the regular octagon, see the beginning of the proof of Proposition \ref{propslope},
we get optimal domains that are non regular octagons. Here we plot two possible optimal domains for the values $x_0=0.985$ and $0.995$.

\begin{center}
\begin{figure}[h]
\includegraphics[height=4cm]{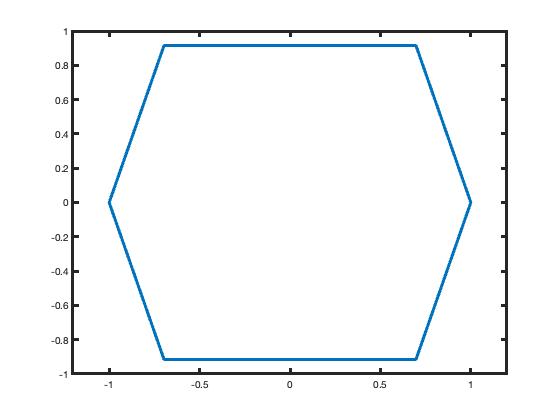}
\includegraphics[height=4cm]{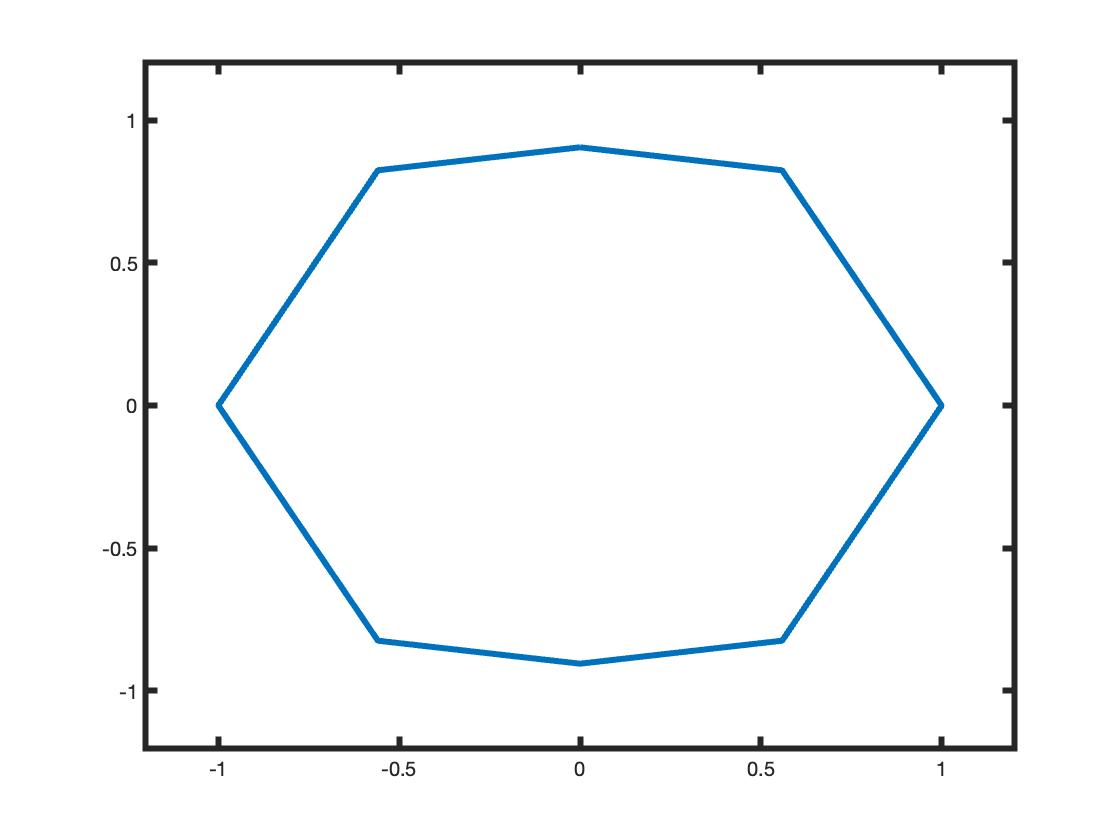}
\caption{A possible optimal domain for $x_0=0.985$ (left), for $x_0=0.995$ (right).}
\end{figure}
\end{center}
\subsection{The Blaschke-Santal\'o diagram}
We plot now the diagram, taking into account all the previous information.
\begin{center}
\begin{figure}[h]
\includegraphics[scale=0.25]{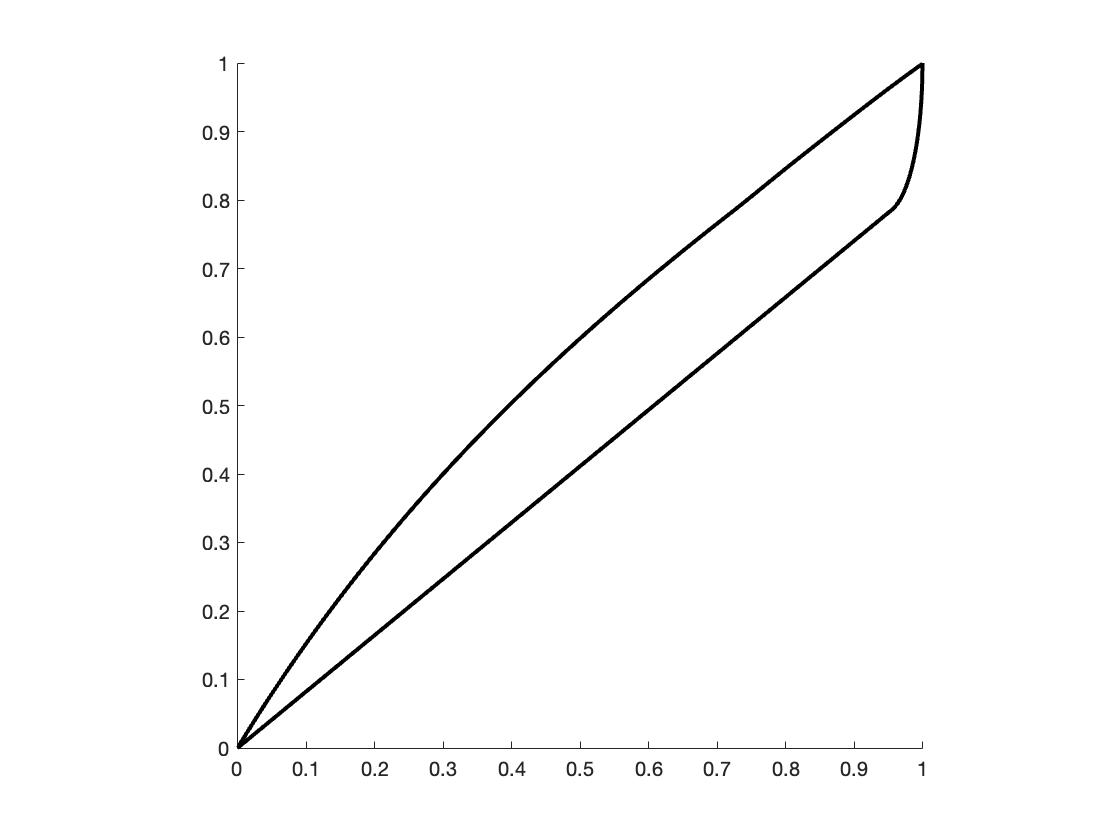}
\caption{The Blaschke-Santal\'o diagram $(A,P,W)$.}
\end{figure}
\end{center}


\begin{thebibliography}{99}


\bibitem{ABo} \textsc{P. Antunes, B. Bogosel}: Parametric shape optimization using the support function,
Comput. Optim. Appl. 82, No. 1, 107-138 (2022). 

\bibitem{AH} {\sc P.~Antunes, A.~Henrot}: On the range of the first two Dirichlet and Neumann eigenvalues of the Laplacian, {\it Proc. R. Soc. Lond. Ser. A} {\bf 467}, 1577--1603 (2011)


\bibitem{VBP}\textsc{M. van den Berg, G. Buttazzo, A. Pratelli}: On the relations between principal eigenvalue and torsional rigidity, \textit{Comm. Contemp. Math.} \textbf{23}, no. 8 (2021)


\bibitem{Bo} \textsc{B. Bogosel}: Numerical shape optimization among convex sets, Appl. Math. Optim. 87, No. 1, Paper No. 1, 31 p. (2023). 

\bibitem{BBO}\textsc{B. Bogosel, G. Buttazzo, E. Oudet}: On the numerical approximation of Blaschke-Santal\'o diagrams using Centroidal Voronoi Tessellations, preprint (2023), arxiv n. 2302.00603


\bibitem{BBF} {\sc D.~Bucur, G.~Buttazzo, I.~Figuereido}: On the attainable eigenvalues of the Laplace operator, {\it SIAM J. Math. Anal.} {\bf 30}, 527--536 (1999)

\bibitem{DHP} \textsc{A. Delyon, A. Henrot, Y. Privat}: The missing (A,D,r) diagram,
{\it Ann. Inst. Fourier} 72, No. 5, 1941-1992 (2022)

\bibitem{FL}\textsc{I. Ftouhi, J. Lamboley}: Blaschke-Santal\'{o} diagram for volume, perimeter and first Dirichlet eigenvalue
Ilias Ftouhi, \textit{ SIAM J. Math. Anal.} \textbf{53}, no. 2 (2021)

\bibitem{HO} \textsc{A. Henrot, E. Oudet}: Minimizing the second eigenvalue of the Laplace operator with Dirichlet boundary conditions,
{\it Arch. Ration. Mech. Anal.} 169, No. 1, 73-87 (2003). 

\bibitem{HP} \textsc{A. Henrot, M. Pierre}: Shape variation and optimization. A geometrical analysis,
EMS Tracts in Mathematics 28. Zürich: European Mathematical Society (EMS).

\bibitem{HC1} \textsc{M.A. Hern\'endez Cifre}:
Is there a planar convex set with given width, diameter, and inradius?
{\it Am. Math. Mon.} 107, No. 10, 893--900 (2000)

\bibitem{HC2} \textsc{M.A. Hern\'endez Cifre, S. Gomis Segura}:  
The missing boundaries of the Santal\'e diagrams for the cases $(d,\omega, R)$ and $(\omega,R,r)$,
{\it Discrete Comput. Geom.} 23, No. 3, 381--388 (2000)


\bibitem{LN} \textsc{J. Lamboley, A. Novruzi}: Polygons as Optimal Shapes with Convexity Constraint, {\it SIAM Journal on Control and Optimization} \textbf{48}, no. 5 (2010), DOI: 10.1137/080738581

\bibitem{LNP} \textsc{J. Lamboley, A. Novruzi, and M. Pierre.}: Regularity and singularities of optimal convex shapes in the plane. , Arch.
Ration. Mech. Anal., 205(1):311–343, 2012.

\bibitem{LP} \textsc{J. Lamboley, R. Prunier}: Regularity in shape optimization under convexity constraint, to appear in
Calculus of Variations and PDE, 2023,  ttps://arxiv.org/pdf/2204.09317.pdf

\bibitem{LZ} \textsc{I. Lucardesi, D. Zucco}: On Blaschke–Santal\'o diagrams for the torsional rigidity and the first Dirichlet eigenvalue, \textit{Annali di Matematica Pura ed Applicata} (2022)

\bibitem{Polya}\textsc{G. P\'olya}: On the role of the circle in certain variational problems,
Ann. Univ. Sci. Budap. Rolando Eötvös, Sect. Math. 3-4, 233-239 (1961). 

\bibitem{Santa} {\sc L. Santal\'o}: Sobre los sistemas completos de desigualdades entre tres elementos de una figura convexa plana, {\it Math. Notae } {\bf 17}, 82--104 (1961)

\bibitem{Sch}  \textsc{R. Schneider}: Convex Bodies: The Brunn-Minkowski Theory. , Cambridge University Press, 2nd expanded edition, 
2013.


\end{thebibliography}
\end{document}